\documentclass[preprint,12pt]{elsarticle} 
\usepackage{geometry} 
\usepackage{amssymb}
\usepackage{amsmath}
\usepackage{amsthm}
\usepackage[capitalize]{cleveref}
\usepackage{blkarray}
\newtheorem{theorem}{Theorem}[section]
\newtheorem{lemma}{Lemma}[section]
\newtheorem{definition}{Definition}[section]
\newtheorem{corollary}{Corollary}[section]
\newtheorem{proposition}{Proposition}

\newtheoremstyle{myremarkstyle}
{3pt} 
{3pt} 
{\normalfont} 
{} 
{\bfseries} 
{.} 
{.5em} 
{} 

\theoremstyle{myremarkstyle}
\newtheorem{remark}{Remark}[section]
\newtheorem{example}{Example}[section]
\numberwithin{equation}{section}

\journal{Advances in Mathematics}

\begin{document}

\begin{frontmatter}



\title{Generalized singular value decompositions of dual quaternion matrices and beyond}


\author[1]{Sitao Ling} 
\ead{lingsitao2004@163.com}

\author[2]{Wenxuan Ma\corref{cor1}}
\ead{marcoper2013@gmail.com}

\author[3]{Musheng Wei}
\ead{mwei@shnu.edu.cn}
\address[1]{School of Mathematics, Jiangsu Center for Applied Mathematics at CUMT,\\
	China University of Mining and Technology, Xuzhou, China}
\address[2]{School of Mathematics, China University of Mining and Technology, Xuzhou, China}
\address[3]{College of Mathematics and Science, Shanghai Normal University, Shanghai China}

\cortext[cor1]{Corresponding author}

\begin{abstract}
In high-dimensional data processing and data analysis related to dual quaternion statistics, generalized singular value decomposition  (GSVD) of a dual quaternion matrix pair is an essential numerical linear algebra tool for an elegant problem formulation and numerical implementation. In this paper, building upon the existing singular value decomposition (SVD) of a dual quaternion matrix, we put forward several types of GSVD of dual quaternion data matrices in accordance with their dimensions. Explicitly, for a given dual quaternion matrix pair $\{{\boldsymbol A}, {\boldsymbol B}\}$, if ${\boldsymbol A}$ and ${\boldsymbol B}$ have the same number of columns, we investigate two forms of their quotient-type SVD (DQGSVD) through different strategies, which can be selected to use in different scenarios. Three artificial examples are presented to illustrate the principle of the DQGSVD.

Alternatively, if ${\boldsymbol A}$ and ${\boldsymbol B}$ have the same number of rows, we consider their canonical correlation decomposition (DQCCD). If ${\boldsymbol A}$ and ${\boldsymbol B}$ are consistent for dual quaternion matrix multiplication, we present their product-type SVD (DQPSVD). As a preparation, we also study the QR decomposition of a dual quaternion matrix based on the dual quaternion Householder transformation, and introduce the CS decomposition of an 2-by-2 blocked unitary dual quaternion matrix. Due to the peculiarity of containing dual part for dual quaternion matrices, the obtained series of GSVD of dual quaternion matrices dramatically distinguish from those in the real number field, the complex number field, and even the quaternion ring, but can be treated as an extension of them to some extent.
\end{abstract}




\end{frontmatter}



\section{Introduction}

Dual quaternion is an extended algebraic structure of dual number that originally introduced by British mathematician William Kingdon Clifford in 1873. With the aid of dual quaternions, the researchers are able to significantly improve the efficiency of computation and optimize graphic rendering and animation effects. In particular, when dealing with 3D rotation and interpolation operations, using dual quaternions can effectively avoid the gimbal lock problem, improve numerical stability, and achieve smooth transition effects. Dual quaternions and the matrices of dual quaternions play an important role in the hand-eye calibration problem \cite{s1}, the simultaneous localization and mapping (SLAM) problem, and the spacecraft position tracking problem. In summary, they have been widely used in various areas \cite{s7,s4,s3,s2,s6,s1}, such as neuroscience, robotics, computer graphics, and multi-agent formation control \cite{QWL2023}. 

Recent emerging researches about dual quaternion matrix problem concentrate on their relevant algebraic properties and numerical computations. For example, Qi, Ling and Yan \cite{s15} proposed a method for describing the magnitude of a dual quaternion and defined the norms of dual quaternion vectors. Qi and Luo \cite{s12} studied the spectral theory of dual quaternion matrices and provided the unitary decomposition of a dual quaternion Hermitian matrix. Later on, Cui and Qi \cite{CJSC2024} derived a power method for computing the dominant eigenvalue of dual quaternion Hermitian matrices. With regard to the general form of a dual quaternion, it contains the infinitesimal part and it does not obey noncommutative law for multiplication, which are the major obstacles for us to set up theoretical analysis of dual quaternion matrices, and some properties about complex matrices or quaternion matrices fall down for dual quaternion matrices.

Low-rank approximation is an efficient method for data compression while dealing with large scale hypercomplex matrices. Singular value decomposition (SVD) is an essential tool for the low-rank approximation problem. The SVD of a quaternion matrix was first introduced in \cite{Zhang1997}. Real or complex structure preserving orthogonal transformation algorithms \cite{Li2014,DZhang2024} and Lanczos iteration algorithm \cite{Jia2019} were designed for its computation. Recently, Liu, Ling and Jia \cite{LLJ2022} presented the randomized SVD algorithm for low-rank quaternion matrix approximation and applied it to color face recognition.  In the dual complex number system, Wei, Ding and Wei \cite{s9} proposed the compact dual singular value decomposition (CDSVD) of dual complex matrices, from which they reported the optimal solution to the best rank-$k$ approximation under a newly defined quasi-metric, and developed an approach to recognize traveling waves in the brain. Based on the SVD of a general dual quaternion matrix \cite{s12}, Ling, He and Qi \cite{s11} established Eckart-Young like theorem for dual quaternion matrices and characterized the optimal low-rank approximation within a given subspace.
Through the SVD of dual quaternion data matrices, we can also afford the tasks of data dimension reduction and feature extraction. This is very helpful for processing large-scale data sets, discovering the main features in the data and improving the computational efficiency.

As a generalization of the SVD for one matrix, the generalized singular value decomposition (GSVD) of a matrix pair is very useful in discriminant analysis for information retrieval systems \cite{Howland}. It can be used in multidimensional machine condition monitoring problems for decision making in a multidimensional case \cite{Cempel}, and DNA copy in biology engineering \cite{Bradley2019}. It can also be applied to solve for the generalized eigenvalue decomposition problem, such a problem arises in certain data science settings where we are interested in extracting the most discriminative information from one dataset of particular interest relative to the other \cite{Chen2019}. The theory of GSVD of quaternion matrices was originally presented in \cite{jiang2006}. Explicitly, for the given ${\mathbf A}$ $\in {\bf \mathbb{Q}}^{m_a\times n}$ and ${\mathbf B}$ $\in {\bf \mathbb{Q}}^{m_b\times n}$ with $m_a\ge n$, there exist unitary matrices ${\mathbf U}$, ${\mathbf V}$ and a nonsingular matrix ${\mathbf X}$ such that 
\begin{equation*}
	\begin{split}
		{\mathbf U}^{*}{\mathbf A}{\mathbf X}&=D_A=\begin{pmatrix}
			I&0&0\\
			0&{\rm diag}(\alpha_{1}, \ldots, \alpha_{s})&0\\
			0&0&0
		\end{pmatrix}, \\
		{\mathbf V}^{*}{\mathbf B}{\mathbf X}&=D_B=\begin{pmatrix}
			I&0&0\\
			0&{\rm diag}(\beta_{1}, \ldots, \beta_{s})&0\\
			0&0&0
		\end{pmatrix},
	\end{split}
\end{equation*}
where $I$ denotes the identity matrix of suitable size, and the real numbers $\alpha_i, \beta_i$ satisfy 
$$1>\alpha_1\ge \alpha_2\ge\cdots\ge \alpha_s>0, \  \  0<\beta_1\le \beta_2\le\cdots\le \beta_s<1,$$ 
$${\alpha_{i}}^2+{\beta_{i}}^2=1, \  i=1,2,\ldots,s.$$
Recently, the joint Lanczos bidiagonalization method with thick-restart was proposed for the partial GSVD of a quaternion matrix pair \cite{ZHu2024}. Obviously, the SVD and the GSVD of quaternion matrices share the same form as those of real matrices. But the SVD of dual quaternion matrices is different from that of real matrices and quaternion matrices \cite{s17}. This prompts us to delve deeper into the question: are there any differences in form between the GSVD of dual quaternion matrices (DQGSVD) and that of real matrices, thereby revealing their unique properties and characteristics within different algebraic structures? However, the current research on the DQGSVD remains unexplored. Therefore, we aim to commence with fundamental principles of matrix decomposition and delve into discussions regarding the DQGSVD. In this paper, building upon the existing SVD of a dual quaternion matrix, we put forward several types of GSVD of dual quaternion data matrices in accordance with their dimensions.

Let us make a short review of different types of GSVD for real matrices.
If the real matrices $A$ and $B$ have the same number of columns, the well known quotient SVD (QSVD) was originally introduced in \cite{s13} and further developed in \cite{page1981, Zha1996}. 
Efficient algorithms for the computation of QSVD contains the CS decomposition plus the Lanczos bidiagonalization process \cite{Zha1996} and the Kogbetliantz algorithm \cite{paige1986} for small and medium scale problem, as well as the joint Lanczos bidiagonalization for large scale problem \cite{h6, h25}. If the matrices $A$ and $B$ are consistent for matrix multiplication, product-type SVD (PSVD) was discussed in \cite{Fernando, Heath}, high relative accuracy Jacobi-type algorithm for the PSVD was developed in \cite{Drmac1998}. The restricted SVD (RSVD) is a simultaneous decomposition of a matrix triplet $\{A, B, C\}$ with compatible dimensions to quasi-diagonal forms. An implicit Kogbetliantz algorithm was proposed for the computation of the RSVD \cite{s20}, followed by Jacobi-type iteration and nonorthonormal transformations \cite{Drmac2000}, and QR-type algorithm based on the CS decomposition \cite{Chu2000}. Recent emerged two neural network models for finding approximations of the GSVD and the RSVD illustrated their efficiency \cite{Zhang2021}.
For the relationship between SVD, GSVD and RSVD we refer to \cite{s19}.

For the given two real matrices $A$ and $B$ having the same number of rows, the canonical correlations between the range spaces spanned by their column vectors have enormous applications, such as neuroscience, machine learning, and bioinformatics \cite{cao2015}. Numerical algorithms using the QR decomposition and SVD together with a first order perturbation analysis for computing the canonical correlations came with the literature \cite{s28}.
Golub and Zha \cite[Theorem 2.1]{GolubZha1994} derived canonical correlation decomposition (CCD) theorem for the given real matrix pair, which can not only explicitly exhibit the canonical correlations of the matrix pair, but also reveal some of its other intrinsic structures. 
Subsequently, the CCD decomposition was well used to find the least-squares solution \cite{XuWei1998} of the matrix equation $AXB+CYD=E$, and the optimal approximate solution in its least-squares solution set \cite{LiaoBai2005}.

In this paper, we study the DQGSVD and relevant theorems over dual quaternion ring.
In \cref{section2}, we introduce some fundamental knowledge about dual quaternions and dual quaternion matrices. 
In \cref{section3}, we present Householder transformation, QR decomposition of dual quaternion matrices and CS decomposition of unitary  dual quaternion matrices.
In \cref{section4}, we introduce two types of the quotient-type SVD for dual quaternion matrices, and further investigate the PSVD, and CCD of dual quaternion matrices.
In \cref{examp5}, we present three artificially toy examples to verify the principle of different quotient-type SVD of the dual quaternion matrices.
In \cref{section6} we conclude the paper by pointing out the differences between two forms of DQGSVD, as well as the differences between the DQSVD and the SVD of complex matrices and quaternion matrices.

\section{Dual quaternions and dual quaternion matrices}\label{section2}

Denote the sets of real numbers, dual numbers, quaternions and dual quaternions by ${\bf \mathbb{R}}$, ${\bf \mathbb{D}}$, ${\bf \mathbb{Q}}$ and ${\bf \mathbb{DQ}}$, respectively. A dual number $\mathsf q$ has the form ${\mathsf q}=q_{st}+q_{in}\epsilon$, where $q_{st},\ q_{in} \in {\bf \mathbb{R}}$ and $\epsilon^{2}=0$. If the standard part $q_{st}$ is nonzero, then $q$ is said to be {\it appreciable}. The conjugate of ${\mathsf q}$ is itself that is denoted by ${\mathsf q}^{*}$.
For a given real coefficient polynomial $F(X)$, evaluate $F$ on a dual number to get $F( a+b\epsilon )=F(a)+F^{'}(a)b\epsilon$, where $F^{'}$ is the derivative of $F$. The multiple of $\epsilon$ is sometimes referred to be {\it infinitesimal} part, which intuitively means that $\epsilon$ is \textquotedblleft so small" that it squares to zero.

As introduced in \cite{s15}, we may define a total order of dual numbers.  Dual numbers ${\mathsf p}=p_{st}+p_{in}\epsilon$ and ${\mathsf q}=q_{st}+q_{in}\epsilon$ $\in {\bf \mathbb{D}}$ with $p_{st},\ q_{st},\ p_{in}, \ q_{in} \in {\bf \mathbb{R}}$, satisfy ${\mathsf p}>{\mathsf q}$, if $p_{st}>q_{st}$ or $p_{st}=q_{st}$ and $p_{in}>q_{in}$; ${\mathsf p}={\mathsf q}$, if and only if $p_{st}=q_{st}$ and $p_{in}=q_{in}$.  We say ${\mathsf q}$ a positive dual number if ${\mathsf q}>0$, and  ${\mathsf q}$ a nonnegative dual number if ${\mathsf q}\ge 0$. For any  positive integer $k$, it is easy to verify the following properties:
\begin{itemize}
	\item[(1)]
	${\mathsf p}+{\mathsf q}=p_{st}+q_{st}+(p_{in}+q_{in})\epsilon$. 
	
	\item[(2)]
	${\mathsf p}{\mathsf q}=p_{st}q_{st}+(p_{st}q_{in}+p_{in}q_{st})\epsilon$.
	
	\item[(3)] 
	${\mathsf q}^{k}=q_{st}^{k}+kq_{st}^{k-1}q_{in}\epsilon$.
	
	\item[(4)]
	If ${\mathsf q}$ is appreciable, then ${\mathsf q}$ is invertible and ${\mathsf q}^{-1}=q_{st}^{-1}-q_{st}^{-1}q_{in}q_{st}^{-1}\epsilon$.
	
	\item[(5)]
	If ${\mathsf q}$ is nonnegative and appreciable, then $\sqrt{{\mathsf q}}=\sqrt{q_{st}}+\frac{q_{in}}{2\sqrt{q_{st}}}\epsilon$.
\end{itemize}

Define the absolute value of ${\mathsf q}$ $\in {\bf \mathbb{D}}$ as 
\begin{equation*}
	\left | {\mathsf q} \right |=\begin{cases}
		\left | q_{st} \right |+{\rm sgn}(q_{st})q_{in}\epsilon& \text{ if } \ q_{st}\neq 0\\
		\left | q_{in} \right |\epsilon& \text{ otherwise }  
	\end{cases}, 
\end{equation*}
where 
\begin{equation*}
	{\rm sgn}(x)=\begin{cases}
		-1,& \text{ if } \ x>0\\
		\ 0,& \text{ if } \ x=0 \\
		\ 1,& \ \text{if} \ \ x<0
	\end{cases}.
\end{equation*}
Then, we have
\begin{itemize}
	\item[(1)]
	$\left | {\mathsf q} \right |=0$ if and only if ${\mathsf q}=0$.
	
	\item[(2)]
	$\left | {\mathsf q} \right | \geq  {\mathsf q}$, and $\left | {\mathsf q} \right | = {\mathsf q}$ holds if ${\mathsf q} \ge 0$.
	
	\item[(3)]
	$\left | {\mathsf q} \right |= \sqrt{{\mathsf q}^{2}}$ if ${\mathsf q}$ is appreciable.
	
	\item[(4)]
	$\left | {\mathsf p}{\mathsf q} \right |=\left | {\mathsf p} \right |\left | {\mathsf q} \right |$.
	
	\item[(5)]
	$\left | {\mathsf p}+{\mathsf q} \right |\le \left | {\mathsf p} \right |+\left | {\mathsf q} \right |$.
\end{itemize}

Throughout the paper, we use bold case letters to denote quaternions or dual quaternions. However, it is easy to determine which kind they belong to from the context, based on whether they are italicized or not.
Donate $\mathbf{q}=q_{0}+q_{1}{\bf i}+q_{2}{\bf j}+q_{3}{\bf k}$ $\in {\bf \mathbb{Q}}$, where $q_{0}$, $q_{1}$, $q_{2}$, $q_{3}$ $\in {\bf \mathbb{R}}$, and $\bf{i}$, $\bf{j}$, $\bf{k}$ are imaginary units satisfying
\begin{equation*}
	\bf{i}^{2}=\bf{j}^{2}=\bf{k}^{2}=\bf{ijk}=-1, \  \   \bf{ij}=-\bf{ji}=\bf{k},  \  \   \bf{jk}=-\bf{kj}=\bf{i}, \  \  \bf{ki}=-\bf{ik}=\bf{j}.
\end{equation*} 
The multiplication of quaternions indeed follows the distribution law, yet it exhibits noncommutative properties, setting it apart from more conventional algebraic structures. The conjugate of a quaternion $\bf q$ is ${\bf q}^{*}=q_{0}-q_{1}{\bf i}-q_{2}{\bf j}-q_{3}{\bf k}$, and the magnitude of $\bf q$ is $\left |{\bf q} \right | =\sqrt{{\bf qq}^*} =\sqrt{q_{0}^{2}+q_{1}^{2}+q_{2}^{2}+q_{3}^{2}}$.

As a derivation of the dual number, a dual quaternion $\boldsymbol{q}$ has the form $\boldsymbol{q}={\bf q}_{st}+{\bf q}_{in}\epsilon$, where ${\bf q}_{st}$ and ${\bf q}_{in}$ are quaternions, and $\epsilon$ is the dual unit satisfying $\epsilon^{2}=0$. $\epsilon$ is commutative for multiplication when it encounters real number, complex number, or quaternion. 
We call ${\bf q}_{st}$ the standard part of $\boldsymbol{q}$, and ${\bf q}_{in}$ the dual part of $\boldsymbol{q}$. If ${\bf q}_{st} \ne 0$, we say that $\boldsymbol{q}$ is appreciable, otherwise, $\boldsymbol{q}$ is infinitesimal. The conjugate of $\boldsymbol{q} $ is denoted as $\boldsymbol{q}^{*}={\bf q}_{st}^{*}+{\bf q}_{in}^{*}\epsilon$. The dual quaternion vector inner product $\left \langle \cdot,\cdot \right \rangle$ in the dual quaternion ring is defined by $\left \langle \boldsymbol{u},\boldsymbol{v} \right \rangle=\boldsymbol{v}^* \boldsymbol{u}$ for any $\boldsymbol{u},\boldsymbol{v}\in {\bf \mathbb{DQ}}^{m}$, from which the induced dual quaternion vector norm is a dual number, defined by $\left \| \boldsymbol{u} \right \|_{2}=\sqrt{\left \langle \boldsymbol{u},\boldsymbol{u} \right \rangle}$.

Ling, He and Qi \cite{s11} defined the orthogonality of two appreciable dual quaternion vectors. In the following definition we make a complement that covers infinitesimal dual quaternion vectors, which is called the weak orthogonality of two general dual quaternion vectors.
\begin{definition}\label{dqorth}
	Let $\boldsymbol{u},\boldsymbol{v}\in {\bf \mathbb{DQ}}^{m}$. Then the dual quaternion vectors $\boldsymbol{u},\boldsymbol{v}$ are weakly orthogonal if $\left \langle \boldsymbol{u},\boldsymbol{v} \right \rangle =0$. Furthermore, if both $\boldsymbol{u}$ and $\boldsymbol{v}$ are appreciable dual quaternion vectors, then they are orthogonal to each other. A more general case, for $k$ appreciable dual quaternion vectors $\boldsymbol{u}_1,\boldsymbol{u}_2,\ldots,\boldsymbol{u}_k\in {\bf \mathbb{DQ}}^{n}$, 
	the $k$-tuple $\{\boldsymbol{u}_1,\boldsymbol{u}_2,\ldots,\boldsymbol{u}_k\}$ is said to be orthogonal if $\left \langle \boldsymbol{u}_{i},\boldsymbol{u}_{j} \right \rangle =0$ for $i\ne j$, and orthonormal if it is orthogonal and further satisfies $\left \langle \boldsymbol{u}_{i},\boldsymbol{u}_{i} \right \rangle =1$ for $i=1,2,\ldots,k$.
\end{definition}

Notice that two orthogonal dual quaternion vectors must be weakly orthogonal, but not vise versa, because the inner product of two infinitesimal dual quaternion vectors is zero obviously.

Denote the collections of $m \times n$ quaternion matrices and dual quaternion  matrices by ${\bf \mathbb{Q}}^{m\times n}$ and ${\bf \mathbb{DQ}}^{m\times n}$, respectively. A dual quaternion matrix $\boldsymbol{A}=({\boldsymbol a}_{ij})\in {\bf \mathbb{DQ}}^{m \times n}$ has the form $\boldsymbol{A}={\bf A}_{st}+{\bf A}_{in}\epsilon$, where ${\bf A}_{st},\ {\bf A}_{in} \in {\bf \mathbb{Q}}^{m\times n}$ are the standard part and the infinitesimal part of $\boldsymbol{A}$, respectively. We say $\boldsymbol{A}$ is appreciable if ${\bf A}_{st} \ne O$; otherwise, $\boldsymbol{A}$ is infinitesimal. As Ling, He and Qi introduced in \cite{s10}, the conjugate transpose of $\boldsymbol{A}$ is denoted as $\boldsymbol{A}^{*}=({\boldsymbol a}_{ji}^{*})={\bf A}_{st}^{*}+{\bf A}_{in}^{*}\epsilon$, and $(\boldsymbol{AB})^{*}=\boldsymbol{B}^{*}\boldsymbol{A}^{*}$ for some $\boldsymbol{B}\in {\bf \mathbb{DQ}}^{n\times p}$. For a square matrix ${\boldsymbol A}\in {\bf \mathbb{DQ}}^{m\times m}$, we say $\boldsymbol{A}$ is nonsingular if there exists $\boldsymbol{AB}=\boldsymbol{BA}=I_{m}$ for some square matrix $\boldsymbol{B}\in {\bf \mathbb{DQ}}^{m\times m}$, and the inverse of $\boldsymbol{A}$ is $\boldsymbol{A}^{-1}=\boldsymbol{B}$. If the square matrix $\boldsymbol{A}\in {\bf \mathbb{DQ}}^{m\times m}$ satisfies $\boldsymbol{A}^{*}=\boldsymbol{A}$, then $\boldsymbol{A}$ is Hermitian. Furthermore, $\boldsymbol{A}$ is unitary if $\boldsymbol{A}^{*}\boldsymbol{A}=\boldsymbol{A}\boldsymbol{A}^{*}=I_m$.

\section{QR and CS decompositions of  dual quaternion matrices}\label{section3}

In order to better understand the GSVD of dual quaternion matrices, we first discuss the Householder transform, QR decomposition and CS decomposition of matrices over dual quaternion ring. These theoretical results are similar but essentially distinguish from those over complex number field. 

\subsection{QR decomposition of a dual quaternion matrix}

First we consider the properties of a dual quaternion Householder transformation matrix. Although it has been well investigated in \cite{s17}, we make a short complement about the eigenvalues of the dual quaternion Householder transformation matrix. 

\begin{proposition}[Householder transformation]\label{lm:3.1}
	Let $\boldsymbol{a}, \boldsymbol{b}\in {\bf \mathbb{DQ}}^{n}$ with $\boldsymbol{a}\ne \boldsymbol{b}$, and $\boldsymbol{H}=I_n-2\boldsymbol{v}\boldsymbol{v}^{*}\in{\bf \mathbb{DQ}}^{n \times n}$ with $\boldsymbol{v}$ being a unit dual quaternion vector. Then the following properties hold:
	\begin{itemize}
		\item[(1)]
		$\boldsymbol{H}^{*}=\boldsymbol{H}$, i.e. $\boldsymbol{H}$ is Hermitian.
		
		\item[(2)]
		$\boldsymbol{H}^{*}\boldsymbol{H}=I_n$, i.e. $\boldsymbol{H}$ is unitary.
		
		\item[(3)]
		$\boldsymbol{H}$ has $n-1$ right eigenvalues $1$, and one right eigenvalue $-1$.
		
		\item[(4)]
		There exists a Householder transformation matrix $\boldsymbol{H}$ such that $\boldsymbol{Ha}=\boldsymbol{b}$, if and only if 
		\begin{equation*}
			\boldsymbol{a}^{*}\boldsymbol{a}=\boldsymbol{b}^{*}\boldsymbol{b},  \quad   \boldsymbol{a}^{*}\boldsymbol{b}=\boldsymbol{b}^{*}\boldsymbol{a}.
		\end{equation*}
		Notice that for any nonzero dual quaternion vector $\boldsymbol{a}=(\boldsymbol{\alpha}_{1}, \ldots, \boldsymbol{\alpha}_{n})$, we have $\boldsymbol{Ha}=-\boldsymbol{\delta} {\left \| \boldsymbol{a} \right \|_{2}}e_{1}$, where
		\begin{equation*}
			\boldsymbol{\delta}=\left\{\begin{matrix}
				\frac{\boldsymbol{\alpha}_{1}}{\left | \boldsymbol{\alpha}_{1} \right |}, & \boldsymbol{\alpha}_{1} \ \text{\rm is appreciable} \\
				1,& \text{\rm otherwise}
			\end{matrix}\right..
		\end{equation*}
	\end{itemize}
\end{proposition}

\begin{proof} 
	The properties (1), (2) and (4) can be found in \cite{s17}. We only prove (3). Obviously, $\boldsymbol{H}\boldsymbol{v}=-\boldsymbol{v}$.
	From \cite[Proposition 3.4 and Proposition 3.11]{s9}, we can expand $\boldsymbol{v}$  to an orthonormal basis of ${\bf \mathbb{DQ}}^{n}$ that is donated by $\widetilde{\boldsymbol{V}}=(\boldsymbol{v}, \boldsymbol{v}_2,\ldots, \boldsymbol{v}_{n})$. From the definition of right eigenvalue of dual quaternion matrices, it is easy to obtain 
	\begin{equation*}
		\boldsymbol{H}\widetilde{\boldsymbol{V}}=\widetilde{\boldsymbol{V}} {\rm diag}(\lambda_{1}, \lambda_{2}, \ldots, \lambda_{n}),
	\end{equation*}
	where $\lambda_{i}\ (i=1, 2, \ldots, n)$ are the right eigenvalues of $\boldsymbol{H}$, and 
	\begin{equation*}
		\lambda_{i}=\left\{\begin{matrix}
			-1,\qquad \quad  i=1 \\
			\  1,  \  \quad {\rm otherwise}
		\end{matrix}\right..
	\end{equation*}
	Then we complete the proof.
\end{proof}

With the Householder transformation of dual quaternion matrices in hand, we deduce the QR decomposition of dual quaternion matrices  in the following.

\begin{theorem}[QR decomposition]\label{QRT}
	Suppose that $\boldsymbol{A}\in {\bf \mathbb{DQ}}^{m \times n}$ with ${\rm rank}(\boldsymbol{A})=r > 0$. Then there exists an $n\times n$ permutation matrix $\Pi$, a unitary matrix $\boldsymbol{Q}\in {\bf \mathbb{DQ}}^{m \times m}$ and an upper trapezoidal dual matrix $\boldsymbol{R}\in {\bf \mathbb{DQ}}^{r \times n}$ such that
	$$\boldsymbol{A}\Pi=
	\boldsymbol{Q}\begin{pmatrix}
		\boldsymbol{R}\\
		0 
	\end{pmatrix},$$
	where the number of non-infinitesimal rows  of $\boldsymbol{R}$ provides the appreciable rank of $\boldsymbol{A}$, denoted by ${\rm Arank}(\boldsymbol{A})$.
\end{theorem}

\begin{proof} 
	We perform the proof by mathematical induction for the rank of $\boldsymbol{A}$. When $r=1$, the conclusion clearly holds. Assuming the conclusion holds for $1\leq r<k$. While $r=k$, there exists a permutation matrix $\Pi_{k}$ such that $\widetilde{\boldsymbol{A}}=\boldsymbol{A}\Pi_{k}$, and the first column $\boldsymbol{a}_{1}$ of $\widetilde{\boldsymbol{A}}$ is a nonzero dual quaternion vector. From \cref{lm:3.1} there exists a Householder transformation $\widetilde{\boldsymbol{Q}}\in {\bf \mathbb{DQ}}^{m \times m}$ such that
	$\widetilde{\boldsymbol{Q}}^{*}\boldsymbol{a}_{1}={\boldsymbol r}_{11} e_{1}$. Then we have
	$$\widetilde{\boldsymbol{Q}}^{*}\widetilde{\boldsymbol{A}}=\begin{pmatrix}
		{\boldsymbol r}_{11}& \ast\\
		0&\boldsymbol{A}_{1} 
	\end{pmatrix},$$
	where $\boldsymbol{A}_{1} \in {\bf \mathbb{DQ}}^{(m-1)\times (n-1)}$ and ${\rm rank}(\boldsymbol{A}_{1})=k-1$. It can be inferred from the inductive assumption that there exists an $(n-1)\times (n-1)$ permutation matrix $\widetilde{\Pi}_{1}$ and a unitary matrix $\widetilde{\boldsymbol{Q}}_{1}\in {\bf \mathbb{DQ}}^{(m-1)\times (m-1)}$ such that
	$$\boldsymbol{A}_{1}\widetilde{\Pi}_{1}=\widetilde{\boldsymbol{Q}}_{1}\begin{pmatrix}
		\boldsymbol{R}_{1}\\
		0 
	\end{pmatrix},$$
	where $\boldsymbol{R}_{1}\in {\bf \mathbb{DQ}}^{(k-1) \times (n-1)}$ is an upper trapezoidal matrix  with ${\rm rank}(\boldsymbol{R}_{1})=k-1$.
	
	Let 
	$$\Pi=\Pi_k\begin{pmatrix}
		1&0\\
		0&\widetilde{\Pi}_{1}
	\end{pmatrix}, \quad
	\boldsymbol{Q}=\widetilde{\boldsymbol{Q}}\begin{pmatrix}
		1&0\\
		0&\widetilde{\boldsymbol{Q}}_{1}
	\end{pmatrix}.$$
	We have 
	$$\boldsymbol{A}\Pi=
	\boldsymbol{Q}\begin{pmatrix}
		\boldsymbol{R}\\
		0 
	\end{pmatrix},$$	
	which completes the proof.
\end{proof}

An alternative form of the QR decomposition of the dual quaternion matrix $\boldsymbol{A}$ in \cref{QRT} is as follows.

\begin{corollary}\label{QRR}
	Suppose that $\boldsymbol{A}\in {\bf \mathbb{DQ}}^{m \times n}$ and ${\rm rank}(\boldsymbol{A})=r > 0$. Then there exists a unitary matrix $\boldsymbol{Q}\in {\bf \mathbb{DQ}}^{m \times m}$ and an upper trapezoidal matrix $\boldsymbol{R}\in {\bf \mathbb{DQ}}^{m \times n}$ such that
	$\boldsymbol{A}=\boldsymbol{Q}\boldsymbol{R}$.
\end{corollary}

From \cref{QRT} we can derive a full-rank decomposition of the dual quaternion matrix $\boldsymbol{A}$ as below.

\begin{corollary}\label{coro:FRD}
	Suppose that $\boldsymbol{A}\in {\bf \mathbb{DQ}}^{m\times n}$ and ${\rm rank}(\boldsymbol{A})=r>0$. Then there exist $\boldsymbol{F}\in {\bf \mathbb{DQ}}^{m\times r}$ and $\boldsymbol{G}\in {\bf \mathbb{DQ}}^{r\times n}$ such that
	$$\boldsymbol{A}=\boldsymbol{F}\boldsymbol{G}.$$
\end{corollary}
\begin{proof}
	In \cref{QRT}, we directly take the first $r$ columns of $\boldsymbol{Q}$ as $\boldsymbol{F}=\boldsymbol{Q}_{1}\in {\bf \mathbb{DQ}}^{m\times r}$ that is of full column rank, and $\boldsymbol{G}=\boldsymbol{R}\Pi^{T}\in {\bf \mathbb{DQ}}^{r\times n}$ is of full row rank.
\end{proof}

\begin{remark}\label{UD}
	The full-rank decomposition of the dual quaternion matrix is not unique.	Taking the first $r$ columns of $\boldsymbol{Q}$ as $\boldsymbol{Q}_{1}$, $\widetilde{\boldsymbol{R}}=\boldsymbol{R}\Pi^{T}\in {\bf \mathbb{DQ}}^{r\times n}$ in \cref{QRT}, we call the decomposition $\boldsymbol{A}=\boldsymbol{Q}_{1}\widetilde{\boldsymbol{R}}$ {\it unitary decomposition} of dual quaternion matrices.
\end{remark}

\subsection{CS decomposition of a unitary dual quaternion matrix}

CS decomposition of a partitioned real orthogonal matrix is one of the important tools in matrix analysis and scientific computing \cite{page1981,s16}. It plays a central role for computing the GSVD of real matrices \cite{{ZBai},Chu2000,1994History}. In this subsection, we study the CS decomposition of a partitioned unitary dual quaternion matrix, which is a preparation for the forthcoming GSVD of unitary dual quaternion matrices.

\begin{lemma}{\label{infiniorth}}
	Suppose $\boldsymbol{u}=\boldsymbol{u}_{in}\epsilon,\boldsymbol{v}=\boldsymbol{v}_{st}+\boldsymbol{v}_{in}\epsilon\in {\bf \mathbb{DQ}}^{n}$. If $\left \langle \boldsymbol{u},\boldsymbol{v} \right \rangle =0$, then $\boldsymbol{u}_{in}$ and $\boldsymbol{v}_{st}$ are orthogonal over quaternion ring.
\end{lemma}
\begin{proof}
	From \cref{dqorth} the dual quaternion vectors $\boldsymbol{u},\boldsymbol{v}$ are orthogonal. Then we have 
	\begin{equation*}
		(\boldsymbol{v}_{st}^{*}+\boldsymbol{v}_{in}^{*}\epsilon)\boldsymbol{u}_{in}\epsilon=0,
	\end{equation*}
	which implies
	\begin{equation*}
		\boldsymbol{v}_{st}^{*}\boldsymbol{u}_{in}=0.
	\end{equation*}
	Then the conclusion holds.
\end{proof}

\begin{lemma}{\label{PreCS}}
	Suppose the dual quaternion matrix $\boldsymbol{A} \in {\bf \mathbb{DQ}}^{m\times n}$ satisfies 
	\begin{equation}\label{AA}
		\boldsymbol{A}^{*}\boldsymbol{A}=\begin{pmatrix}
			0_s&0\\
			0&{\sf \Delta}
		\end{pmatrix},
	\end{equation} 
	where ${\sf \Delta}={\rm diag}({ \sf d}_{1},\ldots,{\sf d}_{t})$ with ${\sf d}_{i}$ $(> 0)$ being the appreciable dual number for $i=1,2,\ldots,t$, and $s+t=n$. Then there exists an $m\times m$ unitary dual quaternion matrix $\boldsymbol{U}_1$ such that 
	\begin{equation*}
		\boldsymbol{U}_1^{*}\boldsymbol{A}=\begin{pmatrix}
			{\bf T}\epsilon&0\\
			0& {\sf \Theta} 
		\end{pmatrix}, 
	\end{equation*}
	where ${\sf \Theta}={\rm diag}({\sf h}_{1},\ldots,{\sf h}_{t})$ with ${\sf h}_{i}$ $(>0)$ being the appreciable dual number for $i=1,2,\ldots,t$, and $\bf T$ is an $(m-t)\times s$ quaternion matrix.
\end{lemma}

\begin{proof}
	The assumption \eqref{AA} implies that the columns of $\boldsymbol{A}$ are weakly orthogonal. More explicitly, the first $s$ columns of $\boldsymbol{A}$ are infinitesimal dual quaternion vectors, and the remaining $n-s$ columns are mutually orthogonal dual quaternion vectors. Starting from column normalization of the last $n-s$ columns of $\boldsymbol{A}$ we obtain a dual quaternion matrix $\hat{\boldsymbol{U}}_2$ with orthonormal columns. We then extend $\hat{\boldsymbol{U}}_2$ by adding a dual quaternion matrix $\hat{\boldsymbol{U}}_1$ with orthonormal columns to form an $m\times m$ unitary dual quaternion matrix $\hat{\boldsymbol{U}}=(\hat{\boldsymbol{U}}_1,\hat{\boldsymbol{U}}_2)$. The columns of $\hat{\boldsymbol{U}}$ afford an orthonormal basis of dual quaternion space ${\bf \mathbb{DQ}}^{m}$. Then we have
	\begin{equation*}
		\hat{\boldsymbol{U}}^{*}\boldsymbol{A}=\begin{pmatrix}
			{\bf T}\epsilon&0\\
			0& {\sf \Theta}
		\end{pmatrix},
	\end{equation*}
	where $\bf T$ is an $(m-t)\times s$ quaternion matrix.
\end{proof}

With regard to the CS decomposition of dual quaternion matrices, we present the following results.

\begin{theorem}[DQCS decomposition]\label{lm:3.3}
	Suppose that $\boldsymbol{W}\in{\bf \mathbb{DQ}}^{n \times n}$ is unitary. Partition $\boldsymbol{W}$ as follows
	\begin{equation*}
		\boldsymbol{W}=\begin{pmatrix}
			\boldsymbol{W}_{11}&\boldsymbol{W}_{12}\\
			\boldsymbol{W}_{21}&\boldsymbol{W}_{22}\\
		\end{pmatrix},
	\end{equation*}
	where $\boldsymbol{W}_{11}\in {\bf \mathbb{DQ}}^{r_1\times t_1}$, $\boldsymbol{W}_{12}\in {\bf \mathbb{DQ}}^{r_1\times t_2}$, $\boldsymbol{W}_{21}\in {\bf \mathbb{DQ}}^{r_2\times t_1}$, $\boldsymbol{W}_{22}\in {\bf \mathbb{DQ}}^{r_2\times t_2}$. Then there exist unitary matrices $\boldsymbol{U}_{1}\in {\bf \mathbb{DQ}}^{r_{1} \times r_{1}}$, $\boldsymbol{U}_{2}\in {\bf \mathbb{DQ}}^{r_{2} \times r_{2}}$, $\boldsymbol{V}_{1}\in {\bf \mathbb{DQ}}^{t_{1} \times t_{1}}$, $\boldsymbol{V}_{2}\in {\bf \mathbb{DQ}}^{t_{2} \times t_{2}}$ such that 
	\begin{equation}\label{DQCSD}
		\begin{split}
			& \begin{pmatrix}
				\boldsymbol{U}_{1}^*&0\\
				0&\boldsymbol{U}_{2}^*
			\end{pmatrix}
			\boldsymbol{W}
			\begin{pmatrix}
				\boldsymbol{V}_{1}&0\\
				0&\boldsymbol{V}_{2}
			\end{pmatrix}
			=\begin{pmatrix}
				{\sf D}_{11}&{\sf D}_{12}\\
				{\sf D}_{21}&{\sf D}_{22}
			\end{pmatrix}  \\
			&=
			\left(\begin{array}{ccccc|ccccc}
				I&&&&&{\Sigma}\epsilon&&&\\
				&I&&&&&0\\
				&&{\sf C}&&&&&{\sf S}&\\
				&&&D\epsilon&&&&&I\\
				&&&&0&&&&&I\\
				\hline
				{\Sigma}\epsilon&&&&&-I&&\\
				&0&&&&&-I\\
				&&{\sf S}&&&&&-{\sf C}&\\
				&&&I&&&&&-D\epsilon\\
				&&&&I&&&&&0
			\end{array}\right),
		\end{split}
	\end{equation}
	where
	$I$ is the identity matrix of suitable size,
	$\Sigma$, $D$ are real diagonal matrices, ${\sf C}, {\sf S}$ are appreciable dual matrices and
	\begin{equation*}
		\begin{split}
			&\Sigma={\rm diag}{(\sigma_{1},\sigma_{2},\ldots,\sigma_{r})},\ \  \sigma_{1}\ge\sigma_{2}\ge\cdots\ge\sigma_{r}>0,\\
			& D ={\rm diag}(d_1,d_2,\ldots\,d_t), \quad d_1\ge d_2\ge\cdots\ge d_t>0, \\
			& {\sf C}={\rm diag}({\sf c}_1,{\sf c}_2,\ldots\,{\sf c}_l),\quad  1>{\sf c}_1\ge {\sf c}_2\ge\cdots\ge {\sf c}_l>0, \\
			& {\sf S} ={\rm diag}({\sf s}_1,{\sf s}_2,\ldots\,{\sf s}_l),\quad  0<{\sf s}_1\le {\sf s}_2\le\cdots\le {\sf s}_l<1, \\
			& {\sf C}^2+{\sf S}^2=I_{l}.
		\end{split}
	\end{equation*}
\end{theorem}

\begin{proof}
	Let $\hat{\boldsymbol{U}}_{1}^*\boldsymbol{W}_{11}\hat{\boldsymbol{V}}_{1}={\sf D}_{11}$ be the DQSVD of $\boldsymbol{W}_{11}$ with $\hat{\boldsymbol{U}}_{1}\in{\bf \mathbb{DQ}}^{r_{1} \times r_{1}}$ and $\hat{\boldsymbol{V}}_{1}\in {\bf \mathbb{DQ}}^{t_{1} \times t_{1}}$ being unitary. Since $\boldsymbol{W}$ is a unitary dual quaternion matrix, $\boldsymbol{W}_{11}^*\boldsymbol{W}_{11}+\boldsymbol{W}_{21}^*\boldsymbol{W}_{21}=I_{t_1}$, $\boldsymbol{W}_{11}^*\boldsymbol{W}_{11}$ and $\boldsymbol{W}_{21}^*\boldsymbol{W}_{21}$ are both $t_1\times t_1$ Hermitian positive semidefinite matrices. Then the eigenvalues of $\boldsymbol{W}_{11}^*\boldsymbol{W}_{11}$ and $\boldsymbol{W}_{21}^*\boldsymbol{W}_{21}$ are all dual numbers within the interval $[0,1]$, and ${\sf D}_{11}$ has the form in \eqref{DQCSD}. Since
	\begin{equation*}
		(\boldsymbol{W}_{21}\hat{\boldsymbol{V}}_{1})^{*}(\boldsymbol{W}_{21}\hat{\boldsymbol{V}}_{1})=I_{r_1}-{\sf D}_{11}^{*}{\sf D}_{11}=
		\begin{pmatrix}
			{0_p}&&\\
			&I_l-{\sf C}^2&\\
			&&I_{t_1-l-p}
		\end{pmatrix},
	\end{equation*}
	the column vectors of $\boldsymbol{W}_{21}\hat{\boldsymbol{V}}_1$ are weakly orthogonal while the first $p$ columns are infinitesimal. From \cref{PreCS} there exists an $r_2\times r_2$ unitary dual quaternion matrix $\boldsymbol{U}_{2}^{(1)}$ such that  
	\begin{equation*}
		(\boldsymbol{U}_{2}^{(1)})^{*}\boldsymbol{W}_{21}\hat{\boldsymbol{V}}_{1}=
		\begin{blockarray}{ccc}
			p & t_{1}-p &  \\
			\begin{block}{(cc)c}
				{\bf T}\epsilon & 0 & r_2-t_1+p \\
				0 & {\sf \Theta} &  t_1-p  \\
			\end{block}\\
		\end{blockarray},
	\end{equation*}
	where ${\bf T}$ is an $(r_2-t_1+p)\times p$ quaternion matrix, and ${\sf \Theta}={\rm diag}({\sf h}_{1},\ldots,{\sf h}_{t_1-p})$ is a dual matrix with ${\sf h}_{i}$ $(>0)$ being the appreciable dual numbers for $i=1,2,\ldots,t_1-p$.
	For the QSVD of $\bf T$, there exist unitary quaternion matrices ${\bf P}\in{\bf \mathbb{Q}}^{(r_2-t_1+p)\times (r_2-t_1+p)}$, ${\bf Q}\in{\bf \mathbb{Q}}^{p\times p}$ such that 
	\begin{equation*}
		{\bf P}^{*}{\bf TQ}=\begin{pmatrix}
			\Sigma& 0\\
			0&0
		\end{pmatrix},
	\end{equation*}
	where $\Sigma={\rm diag}{(\sigma_{1},\sigma_{2},\ldots,\sigma_{r})}$ is a real matrix with $\sigma_{1}\ge\sigma_{2}\ge\cdots\ge\sigma_{r}>0$, and $r={\rm rank}({\bf T})$.
	Let 
	\begin{equation*}
		{\boldsymbol{U}}_{1}=\hat{\boldsymbol{U}}_{1}\begin{pmatrix}
			{\bf Q}&0 \\
			0&I_{r_1-p}
		\end{pmatrix},\quad
		{\boldsymbol{U}}_{2}=\boldsymbol{U}_{2}^{(1)}\begin{pmatrix}
			{\bf P}&0\\
			0&I_{t_1-p}
		\end{pmatrix},\quad{\boldsymbol{V}}_{1}=\hat{\boldsymbol{V}}_{1}\begin{pmatrix}
			{\bf Q}&0\\
			0&I_{t_1-p}
		\end{pmatrix}.
	\end{equation*}
	Then we have ${\boldsymbol{U}}_{1}^{*}\boldsymbol{W}_{11}{\boldsymbol{V}}_{1}={\sf D}_{11}$, ${\boldsymbol{U}}_{2}^{*}\boldsymbol{W}_{21}{\boldsymbol{V}}_{1}={\sf D}_{21}$.
	Similarly, there exists a unitary matrix $\widetilde{\boldsymbol{V}}_{2}\in{\bf \mathbb{DQ}}^{t_{2} \times t_{2}}$ such that ${\boldsymbol{U}}_{1}^*\boldsymbol{W}_{12}\widetilde{\boldsymbol{V}}_{2}=\hat{\boldsymbol D}_{12}$, where $\hat{\boldsymbol D}_{12}={\rm diag}({\bf T}_{1}\epsilon,{\sf S},I,I)$ as given in the following \eqref{DD}. Let ${\boldsymbol{U}}_{2}^*\boldsymbol{W}_{22}\widetilde{\boldsymbol{V}}_{2}=\hat{\boldsymbol D}_{22}$. Then we have
	\begin{equation}\label{DD}
		\begin{split}
			&\begin{pmatrix}
				\widetilde{\boldsymbol{U}}_{1}^*&0\\
				0&\widetilde{\boldsymbol{U}}_{2}^*\\
			\end{pmatrix}
			\begin{pmatrix}
				\boldsymbol{W}_{11}&\boldsymbol{W}_{12}\\
				\boldsymbol{W}_{21}&\boldsymbol{W}_{22}\\
			\end{pmatrix}
			\begin{pmatrix}
				\widetilde{\boldsymbol{V}}_{1}&0\\
				0&\widetilde{\boldsymbol{V}}_{2}\\
			\end{pmatrix}  \\
			&=
			\begin{pmatrix}
				{\sf D}_{11}&\hat{\boldsymbol D}_{12}\\
				{\sf D}_{21}&\hat{\boldsymbol D}_{22}
			\end{pmatrix}  
			=\left(\begin{array}{cccc|cccc}
				I&&&&{\bf T}_{1}\epsilon&&\\
				&{\sf C}&&&&{\sf S}&\\
				&&D\epsilon&&&&I\\
				&&&0&&&&I\\
				\hline
				\begin{pmatrix}
					{\Sigma}\epsilon&0\\
					0&0
				\end{pmatrix}&&&&\boldsymbol{X}_{11}&\boldsymbol{X}_{12}&\boldsymbol{X}_{13}&\boldsymbol{X}_{14}\\
				&{\sf S}&&&\boldsymbol{X}_{21}&\boldsymbol{X}_{22}&\boldsymbol{X}_{23}&\boldsymbol{X}_{24}\\
				&&I&&\boldsymbol{X}_{31}&\boldsymbol{X}_{32}&\boldsymbol{X}_{33}&\boldsymbol{X}_{34}\\
				&&&I&\boldsymbol{X}_{41}&\boldsymbol{X}_{42}&\boldsymbol{X}_{43}&\boldsymbol{X}_{44}
			\end{array}\right),
		\end{split}
	\end{equation}
	which is also a unitary dual quaternion matrix. 
	
	Considering the right-hand side matrix in \eqref{DD}, the column orthogonality between the second, third, fourth block columns and the last four block columns implies $\boldsymbol{X}_{21}, \boldsymbol{X}_{31}, \boldsymbol{X}_{41}$, $\boldsymbol{X}_{32}, \boldsymbol{X}_{42}$, $\boldsymbol{X}_{23}, \boldsymbol{X}_{43}$, $\boldsymbol{X}_{24}, \boldsymbol{X}_{34}, \boldsymbol{X}_{44}$ to be zero matrices and $\boldsymbol{X}_{22}=-{\sf C}, \boldsymbol{X}_{33}=-D\epsilon$. Similarly, the row orthogonality between the second, third, fourth block rows and the fifth block row implies $\boldsymbol{X}_{12}, \boldsymbol{X}_{13}, \boldsymbol{X}_{14}$ to be zero matrices. The identity of the fifth block row gives $\boldsymbol{X}_{11}\boldsymbol{X}_{11}^{*}=I$. Eventually, the row orthogonality between the fourth block row and the first one leads to
	\begin{equation*}
		-{\bf T}_{1}\boldsymbol{X}_{11}^{*}\epsilon
		=\begin{pmatrix}
			{\Sigma}\epsilon&0\\  0&0
		\end{pmatrix}^{T}.
	\end{equation*} 
	
	Let 
	${\boldsymbol{V}}_{2}=\widetilde{\boldsymbol{V}}_{2}{\rm diag}(-\boldsymbol{X}_{11}^{*},I,I,I)$. 
	Then we have the following decomposition
	\begin{equation*}
		\begin{aligned}
			&\begin{pmatrix}
				{\boldsymbol{U}}_{1}^*&0\\
				0&{\boldsymbol{U}}_{2}^*\\
			\end{pmatrix}
			\begin{pmatrix}
				\boldsymbol{W}_{11}&\boldsymbol{W}_{12}\\
				\boldsymbol{W}_{21}&\boldsymbol{W}_{22}\\
			\end{pmatrix}
			\begin{pmatrix}
				{\boldsymbol{V}}_{1}&0\\
				0&\boldsymbol{V}_{2}\\
			\end{pmatrix}\\    
			&=
			\left(\begin{array}{ccccc|ccccc}
				I&&&&&{\Sigma}\epsilon&&\\
				&I&&&&&0\\
				&&{\sf C}&&&&&{\sf S}&\\
				&&&D\epsilon&&&&&I\\
				&&&&0&&&&&I\\
				\hline
				{\Sigma}\epsilon&&&&&-I&&&\\
				&0&&&&&-I\\
				&&{\sf S}&&&&&-{\sf C}&\\
				&&&I&&&&&-D\epsilon&\\
				&&&&I&&&&&0
			\end{array}\right)\\
			&\equiv
			\begin{pmatrix}
				{\sf D}_{11}&{\sf D}_{12}\\
				{\sf D}_{21}&{\sf D}_{22}
			\end{pmatrix},
		\end{aligned}
	\end{equation*}
	which is the equalities given in \eqref{DQCSD}.
\end{proof}

\begin{remark}\label{coroCS}
	Through the proof of \cref{lm:3.3}, if we partition the unitary dual quaternion matrix $\boldsymbol{W}$ to be 2-by-1 blocked matrix, we can obtain the CS decomposition of the 2-by-1 blocked unitary dual quaternion matrix as follows
	\begin{equation*}
		\boldsymbol{W}=\begin{pmatrix}
			\boldsymbol{W}_{11}\\
			\boldsymbol{W}_{21}
		\end{pmatrix}=\begin{pmatrix}
			\boldsymbol{U}_{1}&0\\
			0&\boldsymbol{U}_{2}
		\end{pmatrix}\begin{pmatrix}
			{\sf D}_{11}\\
			{\sf D}_{21}
		\end{pmatrix}\boldsymbol{V}_1^*.
	\end{equation*}
\end{remark}

\begin{remark}\label{rem:3.4}
	If the matrix $\boldsymbol{W}$ in \cref{lm:3.3} is a unitary quaternion matrix, then the blocked row or column containing the infinitesimal elements in \eqref{DQCSD} naturally vanish. As a result, \cref{lm:3.3} is in accordance with the CS decomposition of the unitary quaternion matrix.
\end{remark}

\section{DQGSVD of a dual quaternion matrix pair}\label{section4}

In this section, we put forward several types of GSVD of dual quaternion matrices in accordance with their dimensions.
For the given dual quaternion matrix pair $\{{\boldsymbol A}, {\boldsymbol B}\}$, we investigate two forms of their quotient-type SVD (DQGSVD) which can be selected to use in different scenarios, their product-type SVD (DQPSVD) and canonical correlation decomposition (DQCCD).  We also characterize the relations between these several types GSVD and those in the real number field, the complex number field, and even the quaternion ring.

In order to derive the GSVD of a dual quaternion matrix pair, it is imperative to possess a thorough understanding of the SVD of a dual quaternion matrix. 
\begin{lemma}[\cite{s12}, Theorem 6.1, DQSVD]\label{lm:4.1}
	Suppose that $\boldsymbol{A} \in {\bf \mathbb{DQ}}^{m \times n}$. There exists a unitary dual quaternion matrices $\boldsymbol{U} \in {\bf \mathbb{DQ}}^{m \times m}$ and $\boldsymbol{V} \in {\bf \mathbb{DQ}}^{n \times n}$ such that 
	$$\boldsymbol{U}^{*}\boldsymbol{AV}=\begin{pmatrix}
		{\sf \Sigma}_{t}&0\\
		0&0 
	\end{pmatrix},$$
	where ${\sf \Sigma}_{t}={\rm diag}({\sf \mu}_{1}, \ldots, {\sf \mu}_{r}, \ldots, {\sf \mu}_{t})\in {\bf \mathbb{D}}^{t \times t}$,  $r \le t \le \min \{m,n\}$, ${\sf \mu}_{1} \ge {\sf \mu}_{2} \ge \cdots \ge {\sf \mu}_{r}$ are positive appreciable dual numbers, and ${\sf \mu}_{r+1} \ge {\sf \mu}_{r+2} \ge \cdots \ge {\sf \mu}_{t}$ are positive infinitesimal dual numbers. Counting possible multiplicities of the diagonal entries, the form of ${\sf \Sigma}_{t}$ is unique and $r={\rm Arank}(\boldsymbol{A}),\ t={\rm rank}(\boldsymbol{A})$.
\end{lemma}

Suppose the dual quaternion matrices $\boldsymbol{A}$ and $\boldsymbol{B}$ have the same number of columns. In the following we embark on the derivation of the first quotient-type DQGSVD of the dual quaternion matrix pair $\{\boldsymbol{A}, \boldsymbol{B}\}$, which is closely related to the rank or appreciable rank of the 2-by-1 blocked dual quaternion matrix $\boldsymbol{C}=\begin{pmatrix}
	\boldsymbol{A}\\
	\boldsymbol{B}
\end{pmatrix}$.

\begin{theorem}[DQGSVD1]\label{thm:4.3}
	Let $\boldsymbol{C}=\begin{pmatrix}
		\boldsymbol{A}\\
		\boldsymbol{B}
	\end{pmatrix}\in {\bf \mathbb{DQ}}^{(m+p) \times n}$ with $\boldsymbol{A}\in {\bf \mathbb{DQ}}^{m \times n}$, $\boldsymbol{B}\in {\bf \mathbb{DQ}}^{p \times n}$, and ${\rm rank}(\boldsymbol{C})=k$, ${\rm Arank}(\boldsymbol{C})=t$. There exist two unitary dual quaternion matrices $\boldsymbol{U}\in {\bf \mathbb{DQ}}^{m \times m}, \boldsymbol{V}\in {\bf \mathbb{DQ}}^{p \times p}$ and a dual quaternion matrix $\boldsymbol{X}\in {\bf \mathbb{DQ}}^{n \times n}$ such that 
	\begin{equation}\label{eq:AB1}
		\boldsymbol{A}=\boldsymbol{U}({\sf \Sigma}_{\boldsymbol{A}}, \  0)\boldsymbol{X}, \  \
		\boldsymbol{B}=\boldsymbol{V}({\sf \Sigma}_{\boldsymbol{B}}, \ 0)\boldsymbol{X},
	\end{equation}
	where
	\begin{subequations}\label{eq:SigAB}
		\begin{align}
			&{\sf \Sigma}_{\boldsymbol{A}}=
			\begin{blockarray}{ccccc}
				r & q & t_{1} & k-r-q-t_{1}  \\
				\begin{block}{(cccc)c}
					I&0  &0&0&r\\
					0& \ {\sf S}_{\boldsymbol{A}}  &0&0&q \\
					0&0&\Xi\epsilon&0&t_{1}\\
					0 &0 &0&0&m-r-q-t_{1}\\
				\end{block}
			\end{blockarray},\\ 
			&{\sf \Sigma}_{\boldsymbol{B}}=
			\begin{blockarray}{cccccc}
				r_1 &r-r_1 & q & t_{1} & k-r-q-t_{1}  \\
				\begin{block}{(ccccc)c}
					{\Sigma}\epsilon&0&0&0&0&r_2\\
					0&0&0&0&0&p+r-k-r_2\\
					0&0&\ {\sf S}_{\boldsymbol{B}}   &0&0&q \\
					0&0&0&I&0&t_1\\
					0&0&0&0&I& k-r-q-t_{1}\\
				\end{block}
			\end{blockarray}.
		\end{align}
	\end{subequations}
	$\Sigma$, $\Xi_{\boldsymbol{A}}$ are real diagonal matrices, ${\sf S}_{\boldsymbol{A}}$ and ${\sf S}_{\boldsymbol{B}}$ are appreciable dual matrices and
	\begin{equation*}
		\begin{split}
			\Sigma&={\rm diag}{(\sigma_{1},\sigma_{2},\ldots,\sigma_{r_{1}})},\quad \sigma_{1}\ge\sigma_{2}\cdots\ge\sigma_{r_{1}}>0, \\
			\Xi&={\rm diag}(\xi_1,\xi_2,\ldots,\xi_{t_{1}}), \xi_1\ge \xi_2\ge\cdots\ge \xi_{t_{1}}>0,\\
			{\sf S}_{\boldsymbol{A}} &={\rm diag}({\sf c}_{1}, {\sf c}_{2}, \ldots, {\sf c}_{q}),\ 0<{\sf c}_{1}\le \cdots\le {\sf c}_{q}<1,  \\
			{\sf S}_{\boldsymbol{B}} &={\rm diag}({\sf s}_{1}, {\sf s}_{2}, \ldots, {\sf s}_{q}),\ 1>{\sf s}_{1}\ge\cdots\ge {\sf s}_{q}>0,\\
			&\qquad {\sf c}_{i}^{2}+{\sf s}_{i}^2=1,\ i=1, 2,\ldots, q.
		\end{split}
	\end{equation*}
	More precisely, there exist two unitary dual quaternion matrices $\hat{\boldsymbol{U}}\in {\bf \mathbb{DQ}}^{m \times m}, \hat{\boldsymbol{V}}\in {\bf \mathbb{DQ}}^{p \times p}$ and a nonsingular dual quaternion matrix $\hat{\boldsymbol{X}}\in {\bf \mathbb{DQ}}^{n \times n}$ such that
	\begin{equation}\label{eq:AB2}
		\boldsymbol{A}=\hat{\boldsymbol{U}}(\hat{\sf \Sigma}_{\boldsymbol{A}},  \boldsymbol{N}_{\boldsymbol{A}}\epsilon,0)\hat{\boldsymbol{X}}, \  \
		\boldsymbol{B}=\hat{\boldsymbol{V}}(\hat{{\sf \Sigma}}_{\boldsymbol{B}},\boldsymbol{N}_{\boldsymbol{B}}\epsilon, 0)\hat{\boldsymbol{X}},
	\end{equation}
	where the blocked dual quaternion matrix 
	$\begin{blockarray}{cc}
		s &  \\
		\begin{block}{(c)c}
			\boldsymbol{N}_{\boldsymbol{A}} &  m\\
			\boldsymbol{N}_{\boldsymbol{B}}  &  p  \\
		\end{block}
	\end{blockarray}$ 
	has orthonormal columns,
	\begin{subequations}\label{SigABhat}
		\begin{align}
			&\hat{{\sf \Sigma}}_{\boldsymbol{A}}=
			\begin{blockarray}{ccccc}
				\hat{r} & \hat{q} & \hat{t} & t-\hat{r}-\hat{q}-\hat{t} \\
				\begin{block}{(cccc)c}
					I&0  &0&0&\hat{r}\\
					0& \ \hat{{\sf S}}_{\boldsymbol{A}}  &0&0&\hat{q} \\
					0&0&\hat{\Xi}\epsilon&0&\hat{t}\\
					0 &0 &0&0&m-\hat{r}-\hat{q}-\hat{t}\\
				\end{block}
			\end{blockarray},\\ 
			&\hat{{\sf \Sigma}}_{\boldsymbol{B}}=
			\begin{blockarray}{cccccc}
				\hat{r}_1 &\hat{r}-\hat{r}_1 & \hat{q} & \hat{t} & t-\hat{r}-\hat{q}-\hat{t}  \\
				\begin{block}{(ccccc)c}
					\hat{\Sigma}\epsilon&0&0&0&0&\hat{r}_2\\
					0&0&0&0&0&p+\hat{r}-t-\hat{r}_2\\
					0&0&\  \hat{{\sf S}}_{\boldsymbol{B}}   &0&0&\hat{q} \\
					0&0&0&I&0&\hat{t}\\
					0&0&0&0&I& t-\hat{r}-\hat{q}-\hat{t}\\
				\end{block}
			\end{blockarray}.
		\end{align}
	\end{subequations}
	$\hat{\Sigma}$, $\hat{\Xi}$ are real diagonal matrices, $\hat{{\sf S}}_{\boldsymbol{A}}$ and $\hat{{\sf S}}_{\boldsymbol{B}}$ are appreciable dual matrices and
	\begin{equation*}
		\begin{split}
			\hat{\Sigma}&={\rm diag}{(\hat{\sigma}_{1},\hat{\sigma}_{2},\ldots,\hat{\sigma}_{\hat{r}_{1}})},\quad \hat{\sigma}_{1}\ge \hat{\sigma}_{2}\ge\cdots\ge \hat{\sigma}_{\hat{r}_{1}}>0,\\
			\hat{\Xi}&={\rm diag}(\hat{\xi}_1,\hat{\xi}_2,\ldots\,\hat{\xi}_{\hat{t}}),\ \hat{\xi}_1\ge \hat{\xi}_2\ge\cdots\ge \hat{\xi}_{\hat{t}}>0,\\
			\hat{{\sf S}}_{\boldsymbol{A}} &={\rm diag}(\hat{{\sf c}}_{1}, \hat{{\sf c}}_{2}, \ldots, \hat{{\sf c}}_{\hat{q}}),\ 0<\hat{{\sf c}}_{1}\le \cdots\le \hat{{\sf c}}_{\hat{{q}}}<1,  \\
			\hat{{\sf S}}_{\boldsymbol{B}} &={\rm diag}(\hat{{\sf s}}_{1}, \hat{{\sf s}}_{2}, \ldots, \hat{{\sf s}}_{\hat{q}}),\ 1>\hat{{\sf s}}_{1}\ge\cdots\ge \hat{{\sf s}}_{\hat{q}}>0,\\
			&\qquad \hat{{\sf c}}_{i}^{2}+\hat{{\sf s}}_{i}^2=1,\ i=1,2,\ldots,\hat{q}.
		\end{split}
	\end{equation*}
\end{theorem}
\begin{proof}
	Applying \cref{lm:4.1} to the the DQSVD of $\boldsymbol C$, there exist unitary  dual quaternion matrices $\boldsymbol{P}\in{\bf \mathbb{DQ}}^{(m+p) \times (m+p)}$ and $\boldsymbol{Q}\in{\bf \mathbb{DQ}}^{n\times n}$ such that 
	\begin{equation}\label{SVDP}
		\boldsymbol{P}^{*}\boldsymbol{CQ}=\begin{pmatrix}
			{\sf \Sigma}_{\boldsymbol{C}}&0\\
			0&0
		\end{pmatrix}\equiv 
		\left(\begin{array}{cc|c}
			{\sf \Sigma}_{t}&0&0\\
			0&\Sigma_{s}\epsilon&0\\
			\hline
			0&0&0
		\end{array}\right),
	\end{equation} 
	where ${\sf \Sigma}_{\boldsymbol{C}}={\rm diag}({\sf g}_1, \ldots, {\sf g}_k)$, and the positive dual numbers ${\sf g}_1,\ldots, {\sf g}_k$ are singular values of $\boldsymbol{C}$, $s+t=k$. 
	
	Partition $\boldsymbol{P}$ in \eqref{SVDP} to be the following 2-by-2 blocked matrix  
	\begin{equation}\label{P2blk}
		\boldsymbol{P}=\begin{blockarray}{ccc}
			k &  m+p-k\\
			\begin{block}{(cc)c}
				\boldsymbol{P}_{11}&\boldsymbol{P}_{12}&m  \\
				\boldsymbol{P}_{21}&\boldsymbol{P}_{22}&p  \\
			\end{block}
		\end{blockarray}.
	\end{equation} 
	Considering the first block column of $\boldsymbol{P}$ in \eqref{P2blk}, from \cref{coroCS} we obtain the following DQCS decomposition 	
	\begin{equation*}
		\begin{pmatrix}
			\boldsymbol{P}_{11}\\
			\boldsymbol{P}_{21}
		\end{pmatrix}=
		\begin{pmatrix}
			\boldsymbol{U}&0\\
			0&\boldsymbol{V}
		\end{pmatrix}
		\begin{pmatrix}
			{\sf \Sigma}_{\boldsymbol{A}}\\
			{\sf \Sigma}_{\boldsymbol{B}}
		\end{pmatrix}
		\boldsymbol{W}^*,
	\end{equation*}
	where $\boldsymbol{U}\in{\bf \mathbb{DQ}}^{m \times m}$, $\boldsymbol{V}\in{\bf \mathbb{DQ}}^{p \times p}$ and $\boldsymbol{W}\in{\bf \mathbb{DQ}}^{k \times k}$,
	${\sf \Sigma}_{\boldsymbol{A}}$ and ${\sf \Sigma}_{\boldsymbol{B}}$ are given by \eqref{eq:SigAB}. Then, the first equality of \eqref{SVDP} becomes 
	\begin{equation}\label{csGsvd}
		\begin{pmatrix}
			{\boldsymbol{A}}\\
			{\boldsymbol{B}}
		\end{pmatrix}{\boldsymbol{Q}}=
		\begin{blockarray}{ccc}
			k&   n-k &  \\
			\begin{block}{(cc)c}
				{\boldsymbol{U}}{\sf \Sigma}_{\boldsymbol{A}}\boldsymbol{W}^*{{\sf \Sigma}_{\boldsymbol{C}}}  \  &  \  0  &  m\\
				{\boldsymbol{V}}{\sf \Sigma}_{\boldsymbol{B}}\boldsymbol{W}^*{{\sf \Sigma}_{\boldsymbol{C}}}  \  &  \  0   &  p  \\
			\end{block}
		\end{blockarray}.
	\end{equation}
	By taking \begin{equation}\label{QSVDX1}
		\boldsymbol{X}={\rm diag}(\boldsymbol{W}^{*}{\sf \Sigma}_{\boldsymbol{C}},\,I_{n-k})\boldsymbol{Q}^{*}
	\end{equation} 
	we obtain the decomposition of the form \eqref{eq:AB1}. It is necessary to point out that the nonsingularity of $\boldsymbol{X}$ is without guarantee here. 
	
	Alternatively, partition $\boldsymbol{P}$ in \eqref{SVDP} to be the following 2-by-3 blocked matrix  
	\begin{equation}\label{P3blk}
		\boldsymbol{P}=\begin{blockarray}{cccc}
			t&   s &  m+p-k\\
			\begin{block}{(ccc)c}
				\hat{\boldsymbol{P}}_{11}& \hat{\boldsymbol{P}}_{12}& \hat{\boldsymbol{P}}_{13}&m  \\
				\hat{\boldsymbol{P}}_{21}&  \hat{\boldsymbol{P}}_{22}&  \hat{\boldsymbol{P}}_{23}&p  \\
			\end{block}
		\end{blockarray}.
	\end{equation} 
	From \cref{coroCS} we obtain the DQCS decomposition of the first block column of $\boldsymbol{P}$ in \eqref{P3blk} to be
	\begin{equation}\label{P3CS}
		\begin{pmatrix}
			\hat{\boldsymbol{P}}_{11}\\
			\hat{\boldsymbol{P}}_{21}
		\end{pmatrix}=
		\begin{pmatrix}
			\hat{\boldsymbol{U}}&0\\
			0&\hat{\boldsymbol{V}}
		\end{pmatrix}
		\begin{pmatrix}
			\hat{{\sf \Sigma}}_{\boldsymbol{A}}\\
			\hat{{\sf \Sigma}}_{\boldsymbol{B}}
		\end{pmatrix}
		\hat{\boldsymbol{W}}^*,
	\end{equation} 
	where $\hat{\boldsymbol{U}}\in{\bf \mathbb{DQ}}^{m \times m}$, $\hat{\boldsymbol{V}}\in{\bf \mathbb{DQ}}^{p \times p}$ and $\hat{\boldsymbol{W}}\in{\bf \mathbb{DQ}}^{t \times t}$, $\hat{{\sf \Sigma}}_{\boldsymbol{A}}$ and $\hat{{\boldsymbol \Sigma}}_{\boldsymbol{B}}$ are given by \eqref{SigABhat}.
	Substituting \eqref{P3blk} and \eqref{P3CS} into \eqref{csGsvd} and recalling the explicit form of ${\sf \Sigma}_{\boldsymbol{C}}={\rm diag}({\sf \Sigma}_{t}, \Sigma_{s}\epsilon)$, we have
	\begin{equation*}
		\begin{split}
			\begin{pmatrix}
				{\boldsymbol{A}}\\
				{\boldsymbol{B}}
			\end{pmatrix}{\boldsymbol{Q}}&=
			\begin{pmatrix}
				\hat{\boldsymbol{P}}_{11}{{\sf \Sigma}_{t}}  \  &\hat{\boldsymbol{P}}_{12}{\Sigma}_{s}\epsilon&  \  0\\
				\hat{\boldsymbol{P}}_{21}{{\sf \Sigma}_{t}}  \  & \hat{\boldsymbol{P}}_{22}{\Sigma}_{s}\epsilon&  \  0
			\end{pmatrix}=
			\begin{blockarray}{cccc}
				t&  s & n-k &  \\
				\begin{block}{(ccc)c}
					\hat{\boldsymbol{U}}\hat{{\sf \Sigma}}_{\boldsymbol{A}}\hat{\boldsymbol{W}}^*{\sf \Sigma}_{t}&\hat{\boldsymbol{P}}_{12}\Sigma_{s}\epsilon&0 &  m\\
					\hat{\boldsymbol{V}}\hat{{\boldsymbol \Sigma}}_{\boldsymbol{B}}\hat{\boldsymbol{W}}^*{\sf \Sigma}_{t}&\hat{\boldsymbol{P}}_{22}\Sigma_{s}\epsilon&0 &  p  \\
				\end{block}
			\end{blockarray}\\
			&=\begin{pmatrix}
				\hat{\boldsymbol{U}}& 0\\
				0 &\hat{\boldsymbol{V}}
			\end{pmatrix}
			\begin{pmatrix}
				\hat{{\sf \Sigma}}_{\boldsymbol{A}}\hat{\boldsymbol{W}}^*{\sf \Sigma}_{t}&\hat{\boldsymbol{U}}^*\hat{\boldsymbol{P}}_{12}\Sigma_{s}\epsilon&0\\
				\hat{{\sf{\Sigma}}}_{\boldsymbol{B}}\hat{\boldsymbol{W}}^*{\sf \Sigma}_{t}&\hat{\boldsymbol{V}}^*\hat{\boldsymbol{P}}_{22}\Sigma_{s}\epsilon&0	
			\end{pmatrix}.
		\end{split}
	\end{equation*}
	
	By taking 
	\begin{equation}\label{XXhat}
		\hat{\boldsymbol{X}}={\rm diag}(\hat{\boldsymbol{W}}^*{\sf \Sigma}_{t},\Sigma_{s},I_{n-k})\boldsymbol{Q}^{*},\ \boldsymbol{N}_{\boldsymbol{A}}=\hat{\boldsymbol{U}}^*\hat{\boldsymbol{P}}_{12},\ \boldsymbol{N}_{\boldsymbol{B}}=\hat{\boldsymbol{V}}^*\hat{\boldsymbol{P}}_{22},
	\end{equation} 
	we obtain the decomposition of the form \eqref{eq:AB2}. It is easy to verify that $\hat{\boldsymbol{X}}$ is a nonsingular dual quaternion matrix, and $\boldsymbol{N}_{\boldsymbol{A}}^*\boldsymbol{N}_{\boldsymbol{A}}+\boldsymbol{N}_{\boldsymbol{B}}^*\boldsymbol{N}_{\boldsymbol{B}}=I_s$.
\end{proof}

Distinct from \cref{thm:4.3}, the other quotient-type DQGSVD of the dual quaternion matrix pair $\{\boldsymbol{A}, \boldsymbol{B}\}$ is closely related to the appreciable rank of the 2-by-1 blocked dual quaternion matrix 
$\boldsymbol{C}=\begin{pmatrix}
	\boldsymbol{A}\\
	\boldsymbol{B}
\end{pmatrix}$.

\begin{theorem}[DQGSVD2]\label{DDQGSVD}
	Let	$\boldsymbol{C}=\begin{pmatrix}
		\boldsymbol{A}\\
		\boldsymbol{B}
	\end{pmatrix} \in {\bf \mathbb{DQ}}^{(m+p) \times n} 	$
	with $\boldsymbol{A} \in {\bf \mathbb{DQ}}^{m \times n}, \boldsymbol{B} \in {\bf \mathbb{DQ}}^{p \times n}$, and ${\rm Arank}(\boldsymbol{C})=t$. 
	There exist two unitary dual quaternion matrices $\boldsymbol{U}\in {\bf \mathbb{DQ}}^{m \times m}, \boldsymbol{V}\in {\bf \mathbb{DQ}}^{p \times p}$  and a dual quaternion matrix $\boldsymbol{X}\in {\bf \mathbb{DQ}}^{n \times n}$, such that 
	\begin{equation}\label{eq:AB3}
		\boldsymbol{U}^{*}\boldsymbol{A}\boldsymbol{X}=
		({{\sf \Sigma}_{\boldsymbol{A}}}, \  0), \  \
		\boldsymbol{V}^{*}\boldsymbol{BX}=
		({\sf \Sigma}_{\boldsymbol{B}}, \  0),
	\end{equation}
	where
	\begin{equation*}
		\begin{aligned}
			{\sf \Sigma}_{\boldsymbol{A}}
			&= \begin{blockarray}{ccccc}
				r& q & l  & t-(r+q+l)&   \\
				\begin{block}{(cccc)c}
					I &0  &0 &0  & r  \\
					0&\ {\sf S}_{\boldsymbol{A}}  &0 &0  & q\\
					0&0&{\Xi}\epsilon&0  & l \\
					0&0&0&0  &  m-r-q-l   \\
				\end{block}
			\end{blockarray},  \\
			{{\sf \Sigma}}_{\boldsymbol{B}}&=
			\begin{blockarray}{ccccc}
				r_1& r-r_1 & q & t-(r+q)\\
				\begin{block}{(cccc)c}
					{\Sigma}\epsilon&0&0&0&r_2\\
					0&0&0&0&p-t+r-r_2\\
					0&0&{\sf S}_{\boldsymbol{B}}&0&q\\
					0&0&0&I&t-(r+q)\\
				\end{block}
			\end{blockarray}.
		\end{aligned}
	\end{equation*}
	$\Sigma$, $\Xi$ are real diagonal matrices,  ${\sf S}_{\boldsymbol{A}}$ and ${\sf S}_{\boldsymbol{B}}$ are appreciable dual matrices and
	\begin{equation*}
		\begin{split}
			\Sigma&={\rm diag}{(\sigma_{1},\sigma_{2},\ldots,\sigma_{r_{1}})}, \  \   \sigma_{1}\ge\sigma_{2}\ge\cdots\ge\sigma_{r_{1}}>0,\\
			\Xi &={\rm diag}(\xi_1,\xi_2,\ldots,\xi_l),\ \xi_1\ge \xi_2\ge\cdots\ge \xi_l>0,   \\
			{\sf S}_{\boldsymbol{A}} &={\rm diag}({\sf c}_{1}, {\sf c}_{2}, \ldots, {\sf c}_{q}),\ 0<{\sf c}_{1}\le \cdots\le {\sf c}_{q}<1, \\
			{\sf S}_{\boldsymbol{B}} &={\rm diag}({\sf s}_{1}, {\sf s}_{2}, \ldots, {\sf s}_{q}),\ 1>{\sf s}_{1}\ge\cdots\ge{\sf s}_{q}>0, \\
			&\qquad {\sf c}_{i}^{2}+{\sf s}_{i}^2=1,\ i=1,2,\ldots,q.
		\end{split}
	\end{equation*}
\end{theorem}

\begin{proof}	
	Applying \cref{lm:4.1} we perform the DQSVD of the dual quaternion matrix $\boldsymbol{C}$. There exist unitary dual quaternion matrices $\boldsymbol{P}\in {\bf \mathbb{DQ}}^{(m+p) \times (m+p)}$ and $\boldsymbol{Q}\in {\bf \mathbb{DQ}}^{n \times n}$ such that
	\begin{equation}\label{eq:PCQ}
		\boldsymbol{P}^{*}\boldsymbol{CQ}=
		\begin{pmatrix}
			{\sf \Sigma}_{t}&0  &0\\
			0&{\Sigma}_{s}\epsilon  &0 \\
			0&0  &0\\
		\end{pmatrix}, 
	\end{equation}
	where ${\sf \Sigma}_{t}={\rm diag}({\sf g}_{1}, {\sf g}_{2}, \ldots, {\sf g}_{t})$, ${\sf g}_{i}$ are all  appreciable dual numbers for $i=1, 2,\ldots,t$, and ${\Sigma_{s}}={\rm diag}(\nu_{1}, \nu_{2}, \ldots, \nu_{s})$, $\nu_{j}$ are all positive real numbers for $j=1, 2,\ldots,s,$ and $t={\rm Arank}(\boldsymbol{C})$, $t+s={\rm rank}(\boldsymbol{C})$.
	
	Let $\boldsymbol{Q}=(\boldsymbol{Q}_{1},\ \boldsymbol{Q}_{2},\ \boldsymbol{Q}_{3})$ with $\boldsymbol{Q}_{1}\in{\bf \mathbb{DQ}}^{n \times t}$, $\boldsymbol{Q}_{2}\in{\bf \mathbb{DQ}}^{n \times s}$,
	\begin{align}\label{DQGSVDeq1}
		\boldsymbol{A}_{1}&=\boldsymbol{A}\boldsymbol{Q}_{1}{\sf \Sigma}_{t}^{-1}\in {\bf \mathbb{DQ}}^{m \times t},
	\end{align}
	\begin{align}\label{DQGSVDeq2}
		\boldsymbol{B}_{1}&=\boldsymbol{B}\boldsymbol{Q}_{1}{\sf \Sigma}_{t}^{-1}\in {\bf \mathbb{DQ}}^{p \times t}.
	\end{align}
	
	From \eqref{eq:PCQ} we have $\boldsymbol{A}\boldsymbol{Q}_{3}=\boldsymbol{BQ}_{3}=0$. Performing the DQSVD on $\boldsymbol{A}_1$, there exist unitary dual quaternion matrices $\hat{\boldsymbol{U}}\in{\bf \mathbb{DQ}}^{m \times m}, \boldsymbol{W}\in{\bf \mathbb{DQ}}^{t \times t}$ such that 
	\begin{equation}\label{DQGSVDeq3}
		\begin{split}
			\hat{\boldsymbol{U}}^{*}\boldsymbol{A}_1\boldsymbol{W}
			&= \begin{blockarray}{ccccc}
				r& q & l  & t-(r+q+l)&   \\
				\begin{block}{(cccc)c}
					I&0  &0 &0  & r  \\
					0&\ {\sf S}_{\boldsymbol{A}}  &0 &0  & q\\
					0&0&\Xi\epsilon&0  & l \\
					0&0&0&0  &  m-r-q-l   \\
				\end{block}
			\end{blockarray}  \\
			&\equiv  
			\hat{\sf \Sigma}_{\boldsymbol{A}}.
		\end{split}
	\end{equation}
	
	From \eqref{DQGSVDeq1} and \eqref{DQGSVDeq3} we have 
	\begin{equation}\label{AGsv}
		\hat{\boldsymbol{U}}^{*}\boldsymbol{A}(\boldsymbol{Q}_{1},\ \boldsymbol{Q}_{2},\  \boldsymbol{Q}_{3})\begin{pmatrix}
			{\sf \Sigma}_{t}^{-1}\boldsymbol{W}& 0 &0\\
			0&0  & 0 \\
			0  &0& I_{n-t-s}\\
		\end{pmatrix}=({\hat{\sf \Sigma}_{\boldsymbol{A}}}, \  0).
	\end{equation}
	
	Combining  \eqref{DQGSVDeq1}-\eqref{DQGSVDeq2} with \eqref{eq:PCQ} it is easy to obtain that 	  
	$\boldsymbol{A}_{1}^{*}\boldsymbol{A}_{1}+\boldsymbol{B}_{1}^{*}\boldsymbol{B}_{1}=I_{t}$, further from \eqref{DQGSVDeq3} we have
	\begin{equation*} 
		(\boldsymbol{B}_{1}\boldsymbol{W})^{*}\boldsymbol{B}_{1}\boldsymbol{W}=
		I_{t}-\hat{\sf \Sigma}_{\boldsymbol{A}}^T\hat{\sf \Sigma}_{\boldsymbol{A}}={\rm diag}(\underbrace{0, 0, \ldots, 0}_{r}, \underbrace{1-{\sf c}_{1}^{2}, 1-{\sf c}_{2}^{2},\ldots, 1-{\sf c}_{q}^2}_{q},\underbrace{1,1,\ldots,1}_{t-(r+q)}).
	\end{equation*} 
	This implies that the columns of the matrix $\boldsymbol{B}_{1}\boldsymbol{W}$ are weakly orthogonal. According to \cref{PreCS} we can find an $p\times p$ unitary dual quaternion matrix $\hat{\boldsymbol{V}}$ such that 
	\begin{equation}\label{Bgsv}
		\hat{\boldsymbol{V}}^{*}\boldsymbol{B}_{1}\boldsymbol{W}=\begin{pmatrix}
			{\bf T}\epsilon&0\\
			0&{\sf \Theta}
		\end{pmatrix},
	\end{equation}
	where ${\sf \Theta}={\rm diag}({\sf s}_{1},\ldots,{\sf s}_q,1,\ldots,1)$ with ${\sf s}_{i}$ $(>0)$ being the appreciable dual number for $i=1,2,\ldots,q$, and $\bf T$ is an $(p-t+r)\times r$ quaternion matrix.
	Performing the QSVD of $\bf T$, there exist unitary quaternion matrices $\hat{{\bf P}}\in {\bf {\mathbb{Q}}}^{(p-t+r)\times (p-t+r)}$, $\hat{{\bf Q}}\in {\bf {\mathbb{Q}}}^{r\times r}$ such that
	\begin{equation}\label{Tsvd}
		\hat{{\bf P}}^*{\bf T}\hat{{\bf Q}}=\begin{pmatrix}
			\Sigma&0\\
			0&0
		\end{pmatrix},
	\end{equation}
	where $\Sigma={\rm diag}(\sigma_1,\sigma_2,\ldots,\sigma_{r_1})$ is a real matrix with $\sigma_1\ge \sigma_2 \ge \cdots \ge \sigma_{r_1}>0$, and $r_1={\rm rank}({\bf T})$. Substituting \eqref{Tsvd} into \eqref{Bgsv} we have
	\begin{equation}\label{PB1W}
		\begin{split}
			\begin{pmatrix}
				\hat{{\bf P}}^*&0\\
				0&I_{t-r}
			\end{pmatrix}\hat{\boldsymbol{V}}^{*}\boldsymbol{B}_{1}\boldsymbol{W}\begin{pmatrix}
				\hat{{\bf Q}}&0\\
				0&I_{t-r}
			\end{pmatrix}&=
			\begin{blockarray}{ccccc}
				r_1& r-r_1 & q & t-(r+q)\\
				\begin{block}{(cccc)c}
					{\Sigma}\epsilon&0&0&0&r_2\\
					0&0&0&0&p-t+r-r_2\\
					0&0&{\sf S}_{\boldsymbol{B}}&0&q\\
					0&0&0&I&t-(r+q)\\
				\end{block}
			\end{blockarray}  \\
			&\equiv {\sf \Sigma}_{\boldsymbol{B}},
		\end{split}
	\end{equation}
	where ${\sf S}_{\boldsymbol{B}} ={\rm diag}({\sf s}_{1}, {\sf s}_{2}, \ldots, {\sf s}_{q})$ and its diagonal elements are the appreciable dual numbers satisfying $0<{\sf s}_{1}\le \cdots\le {\sf s}_{q}<1$. 
	
	Let 
	\begin{equation}\label{GsvXU}
		\boldsymbol{X}=\boldsymbol{Q}\begin{pmatrix}
			{\sf \Sigma}_{t}^{-1}\boldsymbol{W}\begin{pmatrix}
				\hat{{\bf Q}}&0\\
				0&I_{t-r}
			\end{pmatrix}& 0 &0\\
			0  & 0 &0\\
			0  &0& I_{n-t-s}\\
		\end{pmatrix},   \ 
		\boldsymbol{U}=\hat{\boldsymbol{U}}
		\begin{pmatrix}
			\hat{{\bf Q}}&0\\
			0&I_{m-r}
		\end{pmatrix},\ 
		\boldsymbol{V}=\hat{\boldsymbol{V}}\begin{pmatrix}
			\hat{{\bf P}}&0\\
			0&I_{t-r}
		\end{pmatrix}. 
	\end{equation}	
	Combining \eqref{PB1W} and \eqref{GsvXU} we derive the explicit decomposition of $\boldsymbol{B}$ in \eqref{eq:AB3}. That is,
	\begin{equation*}
		\begin{aligned}
			\boldsymbol{V}^{*}\boldsymbol{B}\boldsymbol{X}&=
			\begin{pmatrix}
				\hat{{\bf P}}^*&0\\
				0&I_{t-r}
			\end{pmatrix}\hat{\boldsymbol{V}}^{*}\boldsymbol{B}\boldsymbol{Q}\begin{pmatrix}
				{\sf \Sigma}_{t}^{-1}\boldsymbol{W}\begin{pmatrix}
					\hat{{\bf Q}}&0\\
					0&I_{t-r}
				\end{pmatrix}& 0 &0\\
				0  & 0 &0\\
				0  &0& I_{n-t-s}\\
			\end{pmatrix}	  \\
			&=\begin{pmatrix}
				\hat{{\bf P}}^*&0\\
				0&I_{t-r}
			\end{pmatrix}\hat{\boldsymbol{V}}^{*}\boldsymbol{B}(\boldsymbol{Q}_1,\boldsymbol{Q}_2,\boldsymbol{Q}_3)
			\left(\begin{array}{c|cc}
				{\sf \Sigma}_{t}^{-1}\boldsymbol{W}\begin{pmatrix}
					\hat{{\bf Q}}&0\\
					0&I_{t-r}
				\end{pmatrix}& 0 &0\\
				0  & 0 &0\\
				0  &0& I_{n-t-s}\\
			\end{array}\right)	  \\
			&=({\sf \Sigma}_{\boldsymbol{B}},0).
		\end{aligned}
	\end{equation*}
	
	It is turn to derive the explicit decomposition of $\boldsymbol{A}$ in \eqref{eq:AB3}. Combining \eqref{AGsv} and \eqref{GsvXU} we have
	\begin{equation*}
		\begin{split}
			\boldsymbol{U}^{*}\boldsymbol{A}\boldsymbol{X}&=
			\begin{pmatrix}
				\hat{{\bf Q}}^{*}&0\\
				0&I_{m-r}
			\end{pmatrix}\hat{\boldsymbol{U}}^{*}\boldsymbol{A}(\boldsymbol{Q}_{1},\ \boldsymbol{Q}_{2},\ \boldsymbol{Q}_{3})
			\begin{pmatrix}
				{\sf \Sigma}_{t}^{-1}\boldsymbol{W}& 0 &0\\
				0  & 0 &0\\
				0  &0& I_{n-t-s}\\
			\end{pmatrix} 
			\begin{pmatrix}
				\hat{{\bf Q}}&0\\
				0&I_{n-r}
			\end{pmatrix}  \\
			& =\begin{pmatrix}
				\hat{{\bf Q}}^{*}&0\\
				0&I_{m-r}
			\end{pmatrix}(\hat{{\sf \Sigma}}_{\boldsymbol{A}}, \ 0)\begin{pmatrix}
				\hat{{\bf Q}}&0\\
				0&I_{n-r}
			\end{pmatrix} 
			= ({\sf \Sigma}_{\boldsymbol{A}}, 0),
		\end{split}
	\end{equation*}
	where the last equality holds because of the explicit derivation of ${\sf \Sigma}_{\boldsymbol{A}}$ as follows
	\begin{subequations}\label{pai}
		\begin{align}
			{\sf \Sigma}_{\boldsymbol{A}}  
			&=\begin{pmatrix}
				\hat{{\bf Q}}^{*}&0\\
				0&I_{m-r}
			\end{pmatrix}
			\begin{blockarray}{cccc}
				r& q & l  & t-(r+q+l)   \\
				\begin{block}{(c|ccc)}
					I&0  &0 &0    \\    
					0&\ {\sf S}_{\boldsymbol{A}}  &0 &0  \\
					0&0&\Xi\epsilon&0   \\
					0&0&0&0     \\
				\end{block}
			\end{blockarray} 
			\begin{blockarray}{cccc}
				r & q & t-(r+q) \\
				\begin{block}{(ccc)c}
					\hat{\bf Q}&0&0& r\\
					0&I&0 & q \\
					0&0&I & t-(r+q) \\
				\end{block}
			\end{blockarray} \\
			&= \begin{blockarray}{ccccc}
				r& q & l  & t-(r+q+l)&   \\
				\begin{block}{(cccc)c}
					I &0  &0 &0  & r  \\
					0&\ {\sf S}_{\boldsymbol{A}}  &0 &0  & q\\
					0&0&\Xi\epsilon&0  & l \\
					0&0&0&0  &  m-r-q-l   \\
				\end{block}
			\end{blockarray}.
		\end{align}
	\end{subequations}
	Then the conclusion holds.  
\end{proof}

\begin{remark}\label{remark4.4}
	In \cref{thm:4.3} we have applied the DQCS decomposition (\cref{lm:3.3}) of dual quaternion matrices  to derive the DQGSVD1. The factor matrix $\boldsymbol{X}$  in the decomposition \eqref{eq:AB1} is not ensure to be nonsingular. However, the factor matrix $\hat{\boldsymbol{X}}$ in the form \eqref{eq:AB2} is a nonsingular dual quaternion matrix, and there are additional infinitesimal submatrices in the resulting decomposition.  This distinguishes from the GSVD of complex matrix pair \cite{s20}.
	The DQGSVD2 in \cref{DDQGSVD} is derived based on the DQSVD, compared with DQGSVD1, DQGSVD2 directly eliminates the influence of the matrix $\Sigma_s\epsilon$. \cref{thm:4.3} and \cref{DDQGSVD} have provided us several forms of the DQGSVD of dual quaternion matrix pair for their convenient use in various scenarios. This will be showcased by three artificial examples in \cref{examp5}.  
\end{remark}

It is well known that the SVD of a product of two real matrices arises from a problem in control theory involving the computation of system balancing transformations \cite{Heath}. The product-type SVD of a real matrix pair provides a convenient way for computing the SVD of the product of two real matrices (PSVD)  if they are consistent for multiplication.
Our concern here is to provide the product-type SVD of consistent dual quaternion matrices (DQPSVD). At first we establish a foundational decomposition of a 3-by-1 blocked dual quaternion matrix that can provide a strong support for the understanding of DQPSVD.

\begin{lemma}{\label{PrePSVD}}
	Suppose the dual quaternion matrix $\boldsymbol{B} \in {\bf \mathbb{DQ}}^{n \times p}$ having the following blocked form 
	\begin{align*}
		\boldsymbol{B}=
		\begin{blockarray}{cc}
			p &\\ 
			\begin{block}{(c)c}
				\boldsymbol{B}_{1}&r_{1}  \\ 
				\boldsymbol{B}_{2}&r_{2}  \\
				\boldsymbol{B}_{3}&n-r_{1}-r_{2} \\
			\end{block}
		\end{blockarray}.
	\end{align*}
	Then the  dual quaternion matrix $\boldsymbol{B}$ has the following decomposition
	\begin{equation}\label{TPSVD-T}
		\boldsymbol{B}=\boldsymbol{T}{\boldsymbol \Sigma}_{\boldsymbol{B}}\boldsymbol{Y},
	\end{equation}
	and 
	\begin{equation}\label{PPSVD-T}
		\begin{aligned}
			\boldsymbol{T}=
			\begin{blockarray}{cccccc}
				r_{11}&r_{12}&r_{1}-(r_{11}+r_{12})&r_{2}&n-(r_{1}+r_{2}) \\
				\begin{block}{(ccccc)c}
					\boldsymbol{T}_{11}&\boldsymbol{T}_{12}&\boldsymbol{T}_{13}&0&0&r_{1} \\
					\boldsymbol{T}_{21}&0&0&\boldsymbol{U}_{1}&0&r_{2}\\
					\boldsymbol{T}_{31}&0&0&0&\boldsymbol{U}_{2}&	n-({r_{1}+r_{2}})\\
				\end{block}
			\end{blockarray}
		\end{aligned},
	\end{equation}
	\begin{equation}\label{PPSVD-B}
		\begin{aligned}
			{\boldsymbol \Sigma}_{\boldsymbol{B}}=
			\begin{blockarray}{ccc}
				r_{11}&p-r_{11}\\
				\begin{block}{(cc)c}
					{\sf \Sigma}_{\boldsymbol{B}}^{1}&0&r_{11}\\
					0&\begin{pmatrix}
						\hat{\Sigma}_{\boldsymbol{B}}\epsilon&0\\
						0&0
					\end{pmatrix}\boldsymbol{G}^{-1}&r_{1}-r_{11}\\
					0&\begin{pmatrix}
						{\boldsymbol \Sigma}_{\boldsymbol{B}}^{2}\\
						{\boldsymbol \Sigma}_{\boldsymbol{B}}^{3}
					\end{pmatrix}&n-r_{1}\\
				\end{block}
			\end{blockarray}
		\end{aligned},
	\end{equation}
	where the submatrix 
	$(\boldsymbol{T}_{11},\boldsymbol{T}_{12},\boldsymbol{T}_{13})\in {\bf \mathbb{DQ}}^{r_{1} \times r_{1}},$
	$\boldsymbol{U}_{1}\in {\bf \mathbb{DQ}}^{r_{2} \times r_{2}},$
	$\boldsymbol{U}_{2}\in {\bf \mathbb{DQ}}^{(n-r_{1}-r_{2}) \times (n-r_{1}-r_{2})}$ are unitary matrices, 
	$\boldsymbol{G} \in {\bf \mathbb{DQ}}^{(p-r_{11})\times (p-r_{11})}$, 
	$\boldsymbol{Y}\in {\bf \mathbb{DQ}}^{p \times p}$ are nonsingular matrices given by \eqref{LEMV} and
	\begin{equation}\label{SigB2B3}
		{\boldsymbol \Sigma}_{\boldsymbol{B}}^{2}=\begin{pmatrix}
			{\sf \Sigma}_1, &\boldsymbol{N}_{1}\epsilon, &0
		\end{pmatrix},\ 
		{\boldsymbol \Sigma}_{\boldsymbol{B}}^{3}=\begin{pmatrix}
			{\sf \Sigma}_2, &\boldsymbol{N}_{2}\epsilon, &0
		\end{pmatrix}.
	\end{equation}
	The blocked dual quaternion matrix 
	$\begin{pmatrix}
		\boldsymbol{N}_{1}\\
		\boldsymbol{N}_{2}
	\end{pmatrix}$ has orthonormal columns, and
	\begin{equation*}
		\begin{aligned}
			&{\sf \Sigma}_{1}=
			\begin{blockarray}{ccccc}
				r & q & t & k-r-q-t  \\
				\begin{block}{(cccc)c}
					I_{1}&0  &0&0&r\\
					0&\ {\sf S}_{1}  &0&0 &q\\
					0&0&\Xi\epsilon&0&t\\
					0 &0 &0&0&r_2-r-q-t\\
				\end{block}
			\end{blockarray},\\ 
			&{\sf \Sigma}_{2}=
			\begin{blockarray}{cccccc}
				\hat{r}_1 &r-\hat{r}_1 & q & t & k-r-q-t  \\
				\begin{block}{(ccccc)c}
					{\Sigma}\epsilon&0  &0&0&0&\hat{r}_2\\
					0&0&0&0&0&n-(r_{1}+r_{2})+r-k-\hat{r}_2\\
					0&0&\ {\sf S}_{2}   &0&0&q \\
					0&0&0&I&0&t\\
					0&0&0&0&I &k-r-q-t\\
				\end{block}
			\end{blockarray}.
		\end{aligned}
	\end{equation*}
	$\Sigma$, $\Xi$ and $\hat{\Sigma}_{\boldsymbol{B}}$ are real diagonal matrices, ${\sf S}_{1}$, ${\sf S}_{2}$ and ${\sf \Sigma}_{\boldsymbol{B}}^{1}$ are appreciable dual matrices and
	\begin{equation*}
		\begin{split}
			\Sigma&={\rm diag}{(\sigma_{1},\sigma_{2},\ldots,\sigma_{\hat{r}_{1}})},    \  \sigma_{1}\ge\sigma_{2}\ge\cdots\ge\sigma_{\hat{r}_{1}}>0,\\
			{\sf \Sigma}_{\boldsymbol{B}}^{1}&={\rm diag}({\sf d}_{1},{\sf d}_{2},\ldots,{\sf d}_{r_{11}}),\ {\sf d}_{1}\ge {\sf d}_{2}\ge\cdots\ge {\sf d}_{r_{11}}>0,\\
			\hat{\Sigma}_{\boldsymbol{B}}&={\rm diag}{(\mu_{r_{11}+1},\mu_{r_{11}+2},\ldots,\mu_{r_{12}})},\ \mu_{r_{11}+1}\ge\mu_{r_{11}+2}\ge\cdots\ge\mu_{r_{12}}>0,\\
			\Xi&={\rm diag}(\xi_1,\xi_2,\ldots\,\xi_t),\ \xi_1\ge \xi_2\ge\cdots\ge \xi_t>0,\\
			{\sf S}_{1} &={\rm diag}({\sf c}_{1}, {\sf c}_{2}, \ldots, {\sf c}_{q}),\ 0<{\sf c}_{1}\le \cdots\le {\sf c}_{q}<1,  \\
			{\sf S}_{2} &={\rm diag}({\sf s}_{1}, {\sf s}_{2}, \ldots, {\sf s}_{q}),\ 1>{\sf s}_{1}\ge\cdots\ge {\sf s}_{q}>0,\\
			&\qquad {\sf c}_{i}^{2}+{\sf s}_{i}^2=1,\ i=1,2,\ldots,q.
		\end{split}
	\end{equation*}
\end{lemma}

\begin{proof}
	From DQSVD of the $r_1\times p$ dual quaternion matrix $\boldsymbol{B}_{1}$, there exist unitary matrices $\boldsymbol{T}_{1} \in {\bf \mathbb{DQ}}^{r_{1} \times r_{1}} $ and $\boldsymbol{W} \in {\bf \mathbb{DQ}}^{p \times p}$ such that 
	\begin{equation}\label{B1SVD}
		\boldsymbol{B}_{1}=\boldsymbol{T}_{1}\begin{pmatrix}
			{\sf \Sigma}_{\boldsymbol{B}}^{1}&0&0\\
			0&\hat{\Sigma}_{\boldsymbol{B}}\epsilon&0\\
			0&0&0
		\end{pmatrix}\boldsymbol{W}^{*}=(\boldsymbol{T}_{11},\ \boldsymbol{T}_{12},\ \boldsymbol{T}_{13})\begin{pmatrix}
			{\sf \Sigma}_{\boldsymbol{B}}^{1}&0&0\\
			0&\hat{\Sigma}_{\boldsymbol{B}}\epsilon&0\\
			0&0&0
		\end{pmatrix}
		\begin{blockarray}{cc}
			p& \\
			\begin{block}{(c)c}
				\boldsymbol{W}^{*}_{11} & r_{11}\\
				\boldsymbol{W}^{*}_{12}  &  r_{12}\\
				\boldsymbol{W}^{*}_{13} & p- r_{11} -r_{12}\\
			\end{block}
		\end{blockarray},
	\end{equation}
	where we have used the partitioned forms 
	\begin{equation*}
		\boldsymbol{T}_{1}=\begin{blockarray}{cccc}
			r_{11} &  r_{12} & r_1- r_{11} -r_{12} \\
			\begin{block}{(ccc)c}
				\boldsymbol{T}_{11} & \boldsymbol{T}_{12} & \boldsymbol{T}_{13} & n_1\\
			\end{block}
		\end{blockarray}, \  \
		\boldsymbol{W}=\begin{blockarray}{cccc}
			r_{11} &  r_{12} & p- r_{11} -r_{12} \\
			\begin{block}{(ccc)c}
				\boldsymbol{W}_{11} & \boldsymbol{W}_{12} & \boldsymbol{W}_{13} & p  \\
			\end{block}
		\end{blockarray}.
	\end{equation*} 
	Then the dual quaternion matrix $\boldsymbol{B}$ can be rewritten as 
	\begin{align}\label{BB}
		\boldsymbol{B}=
		\begin{pmatrix}
			\boldsymbol{B}_{1}\\
			\boldsymbol{B}_{2}\\
			\boldsymbol{B}_{3}
		\end{pmatrix}&=\left(\begin{array}{ccc|cc}
			\boldsymbol{T}_{11}&\boldsymbol{T}_{12}&\boldsymbol{T}_{13}&0&0\\
			\hline
			\boldsymbol{B}_{2}\boldsymbol{W}_{11}({\sf \Sigma}_{\boldsymbol{B}}^{1})^{-1}&0&0&I_{m_{2}}&0\\
			\boldsymbol{B}_{3}\boldsymbol{W}_{11}({\sf \Sigma}_{\boldsymbol{B}}^{1})^{-1}&0&0&0&I_{m_{3}}
		\end{array}\right)\\
		&\quad \cdot 
		\begin{pmatrix}
			{\sf \Sigma}_{\boldsymbol{B}}^{1}&0&0\\
			0&\hat{\Sigma}_{\boldsymbol{B}}\epsilon&0\\	
			0&0&0\\  \hline
			0&\boldsymbol{B}_{2}\boldsymbol{W}_{12}&\boldsymbol{B}_{2}\boldsymbol{W}_{13}\\
			0&\boldsymbol{B}_{3}\boldsymbol{W}_{12}&\boldsymbol{B}_{3}\boldsymbol{W}_{13}
		\end{pmatrix}\begin{pmatrix}
			\boldsymbol{W}^{*}_{11}\\
			\boldsymbol{W}^{*}_{12}\\
			\boldsymbol{W}^{*}_{13}
		\end{pmatrix}.
	\end{align}
	
	According to \cref{thm:4.3}, with the DQGSVD of the $(n-r_{1})\times (p-r_{11})$ dual quaternion submatrix 
	$\begin{pmatrix}
		\boldsymbol{B}_{2}\boldsymbol{W}_{12}&\boldsymbol{B}_{2}\boldsymbol{W}_{13}\\
		\boldsymbol{B}_{3}\boldsymbol{W}_{12}&\boldsymbol{B}_{3}\boldsymbol{W}_{13}
	\end{pmatrix}$, there exist two unitary matrices $\boldsymbol{U}_{1}\in {\bf \mathbb{DQ}}^{r_{2} \times r_{2}}$,  $\boldsymbol{U}_{2} \in {\bf \mathbb{DQ}}^{(n-r_{1}-r_{2})\times (n-r_{1}-r_{2})}$ and a nonsingular matrix $\boldsymbol{G} \in {\bf \mathbb{DQ}}^{(p-r_{11})\times (p-r_{11})}$
	such that 
	\begin{equation}\label{gsvB23}
		\begin{pmatrix}
			\boldsymbol{B}_{2}\boldsymbol{W}_{12}&\boldsymbol{B}_{2}\boldsymbol{W}_{13}\\ \hline
			\boldsymbol{B}_{3}\boldsymbol{W}_{12}&\boldsymbol{B}_{3}\boldsymbol{W}_{13}
		\end{pmatrix} =
		\begin{pmatrix}
			\boldsymbol{U}_1 & 0 \\ 0  &  \boldsymbol{U}_2
		\end{pmatrix}
		\begin{pmatrix}
			{\boldsymbol \Sigma}_{\boldsymbol{B}}^{2}\\
			{\boldsymbol \Sigma}_{\boldsymbol{B}}^{3}
		\end{pmatrix}
		\boldsymbol{G},
	\end{equation}
	where ${\boldsymbol \Sigma}_{\boldsymbol{B}}^{2}, {\boldsymbol \Sigma}_{\boldsymbol{B}}^{3}$ are given by \eqref{SigB2B3}.
	Substituting \eqref{gsvB23} into \eqref{BB} we rewrite $\boldsymbol{B}$ to be
	\begin{align*}
		\boldsymbol{B}=\begin{pmatrix}
			\boldsymbol{B}_{1}\\
			\boldsymbol{B}_{2}\\
			\boldsymbol{B}_{3}
		\end{pmatrix}&=\left(\begin{array}{ccc|cc}
			\boldsymbol{T}_{11}&\boldsymbol{T}_{12}&\boldsymbol{T}_{13}&0&0\\  \hline
			\boldsymbol{B}_{2}\boldsymbol{W}_{11}({\sf \Sigma}_{\boldsymbol{B}}^{1})^{-1}&0&0&\boldsymbol{U}_{1}&0\\
			\boldsymbol{B}_{3}\boldsymbol{W}_{11}({\sf \Sigma}_{\boldsymbol{B}}^{1})^{-1}&0&0&0&\boldsymbol{U}_{2}
		\end{array}\right)  \\
		&\cdot\begin{pmatrix}
			{\sf \Sigma}_{\boldsymbol{B}}^{1}&0\\
			0&\begin{pmatrix}
				\hat{\Sigma}_{\boldsymbol{B}}\epsilon&0\\
				0&0
			\end{pmatrix}\boldsymbol{G}^{-1}\\  \hline
			0&\begin{pmatrix}
				{\boldsymbol \Sigma}_{\boldsymbol{B}}^{2}\\
				{\boldsymbol \Sigma}_{\boldsymbol{B}}^{3}
			\end{pmatrix}\\
		\end{pmatrix}
		\begin{pmatrix}
			I_{r_{11}} & 0  \\  0  &  \boldsymbol{G}
		\end{pmatrix}
		\begin{pmatrix}
			\boldsymbol{W}_{11}^{*}\\
			\boldsymbol{W}_{12}^{*}\\
			\boldsymbol{W}_{13}^{*}
		\end{pmatrix}.
	\end{align*}
	Let  $\boldsymbol{T}$ and ${\boldsymbol \Sigma}_{\boldsymbol{B}}$ be as given in  \eqref{PPSVD-T}-\eqref{PPSVD-B}, and
	\begin{equation}\label{LEMV}
		\boldsymbol{Y}=\begin{pmatrix}
			I_{r_{11}} & 0  \\  0  &  \boldsymbol{G}
		\end{pmatrix}
		\begin{pmatrix}
			\boldsymbol{W}_{11}^{*}\\
			\boldsymbol{W}_{12}^{*}\\
			\boldsymbol{W}_{13}^{*}
		\end{pmatrix}
		=\begin{pmatrix}
			\boldsymbol{W}_{11}^{*}\\
			\boldsymbol{G}\begin{pmatrix}
				\boldsymbol{W}_{12}^{*}\\
				\boldsymbol{W}_{13}^{*}
			\end{pmatrix}
		\end{pmatrix}.
	\end{equation}
	Then $\boldsymbol{Y}$ is an $p\times p$ nonsingular dual quaternion matrix, 
	and the decomposition in \eqref{TPSVD-T} is derived.
\end{proof}

Based on \cref{PrePSVD} we present the following product-type DQGSVD of a dual quaternion matrix pair $\{\boldsymbol{A}, \boldsymbol{B}\}$.

\begin{theorem}[DQPSVD]
	Let $\boldsymbol{A}\in {\bf \mathbb{DQ}}^{m\times n}$, $\boldsymbol{B}\in {\bf \mathbb{DQ}}^{n\times p}$. Then there exist a unitary matrix $\boldsymbol{U}\in {\bf \mathbb{DQ}}^{m\times m}$ and two nonsingular matrices $\boldsymbol{X}\in {\bf \mathbb{DQ}}^{n\times n}$, $\boldsymbol{Y}\in {\bf \mathbb{DQ}}^{p\times p}$ such that 
	\begin{equation}\label{pgsv}
		\boldsymbol{A}=\boldsymbol{U}{\sf D}_{\boldsymbol{A}}\boldsymbol{X}^{-1},\ \boldsymbol{B}=\boldsymbol{X}{\boldsymbol D}_{\boldsymbol{B}}\boldsymbol{Y},
	\end{equation}
	with
	\begin{equation*}
		\begin{aligned}
			{\sf D}_{\boldsymbol{A}}&=
			\begin{blockarray}{cccc}
				r_{1} &r_{2}&n-(r_{1}+r_{2})\\ 
				\begin{block}{(ccc)c}
					I&0&0&r_{1}  \\ 
					0&I\epsilon&0&r_{2}  \\
					0&0&0&m-(r_{1}+r_{2}) \\
				\end{block}
			\end{blockarray}, \\
			{\boldsymbol D}_{\boldsymbol{B}}&=
			\begin{blockarray}{ccc}
				r_{11}&p-r_{11}\\
				\begin{block}{(cc)c}
					{\boldsymbol \Sigma}_{\boldsymbol{B}}^{1}&0&r_{11}\\
					0&\begin{pmatrix}
						\hat{\Sigma}_{\boldsymbol{B}}\epsilon&0\\
						0&0
					\end{pmatrix}\boldsymbol{G}^{-1}&r_{1}-r_{11}\\
					0&\begin{pmatrix}
						{\boldsymbol \Sigma}_{\boldsymbol{B}}^{2}\\
						{\boldsymbol \Sigma}_{\boldsymbol{B}}^{3}
					\end{pmatrix}&n-r_1   \\
				\end{block}
			\end{blockarray},
		\end{aligned}
	\end{equation*}
	where $r_{1}={\rm Arank}(\boldsymbol{A})$, $r_{1}+r_{2}={\rm rank}(\boldsymbol{A})$,  and
	$\boldsymbol{U}$, $\boldsymbol{X}$, $\boldsymbol{Y}$ are given by \eqref{QU}, \eqref{X1T} and \eqref{LEMV}, respectively. $\boldsymbol{G} \in {\bf \mathbb{DQ}}^{(p-r_{11})\times (p-r_{11})}$ is a nonsingular matrix given by \eqref{LEMV},
	\begin{equation*}
		{\boldsymbol \Sigma}_{\boldsymbol{B}}^{2}=\begin{pmatrix}
			{\sf \Sigma}_1, &\boldsymbol{N}_{1}\epsilon, &0
		\end{pmatrix},\ 
		{\boldsymbol \Sigma}_{\boldsymbol{B}}^{3}=\begin{pmatrix}
			{\sf \Sigma}_2, &\boldsymbol{N}_{2}\epsilon, &0
		\end{pmatrix}.
	\end{equation*}
	The blocked dual quaternion matrix $\begin{pmatrix}
		\boldsymbol{N}_{1}\\
		\boldsymbol{N}_{2}
	\end{pmatrix}$ has orthonormal columns, and
	\begin{equation*}
		\begin{aligned}
			&{\sf \Sigma}_{1}=
			\begin{blockarray}{ccccc}
				r & q & t & k-r-q-t  \\
				\begin{block}{(cccc)c}
					I_{1}&0  &0&0&r\\
					0&\ {\sf S}_{1}  &0&0 &q\\
					0&0&\Xi\epsilon&0&t\\
					0 &0 &0&0&n_2-r-q-t\\
				\end{block}
			\end{blockarray},\\ 
			&{\sf \Sigma}_{2}=
			\begin{blockarray}{cccccc}
				{\hat r}_1 &r-{\hat r}_1 & q & t & k-r-q-t  \\
				\begin{block}{(ccccc)c}
					{\Sigma}\epsilon&0  &0&0&0&{\hat r}_2\\
					0&0&0&0&0&n-r_1-r_2+r-k-{\hat r}_2\\
					0&0&\ {\sf S}_{2}   &0&0&q \\
					0&0&0&I&0&t\\
					0&0&0&0&I &k-r-q-t\\
				\end{block}
			\end{blockarray}.
		\end{aligned}
	\end{equation*}
	$\Sigma$, $\Xi$ are real diagonal matrices, ${\sf S}_{1}$ and ${\sf S}_{2}$ are appreciable dual matrices and 
	\begin{equation*}
		\begin{split}
			\Sigma&={\rm diag}{(\sigma_{1},\sigma_{2},\ldots,\sigma_{\hat{r}_{1}})},\quad \sigma_{1}\ge\sigma_{2}\ge \cdots  \ge \sigma_{\hat{r}_{1}}>0,\\
			\Xi&={\rm diag}(\xi_1,\xi_2,\ldots,\xi_t),\  \xi_1\ge \xi_2\ge\cdots\ge \xi_t>0,\\
			{\sf S}_{1} &={\rm diag}({\sf c}_{1}, {\sf c}_{2}, \ldots, {\sf c}_{q}),\ 0<{\sf c}_{1}\le \cdots\le {\sf c}_{q}<1,  \\
			{\sf S}_{2} &={\rm diag}({\sf s}_{1}, {\sf s}_{2}, \ldots, {\sf s}_{q}),\ 1>{\sf s}_{1}\ge\cdots\ge {\sf s}_{q}>0,\\
			&\qquad {\sf c}_{i}^{2}+{\sf s}_{i}^2=1,\ i=1,2,\ldots,q.
		\end{split}
	\end{equation*}	
\end{theorem}

\begin{proof}
	With the DQSVD of the dual quaternion matrix $\boldsymbol{A}$, there exist unitary matrices $\widetilde{\boldsymbol{U}} \in {\bf \mathbb{DQ}}^{m\times m}$, $\widetilde{\boldsymbol{W}}=(\widetilde{\boldsymbol{W}}_{1},\widetilde{\boldsymbol{W}}_{2},\widetilde{\boldsymbol{W}}_{3})\in {\bf \mathbb{DQ}}^{n\times n}$ such that
	\begin{equation}\label{DSVDA}
		\boldsymbol{A}=\widetilde{\boldsymbol{U}}\begin{pmatrix}
			{\sf \Sigma}_{\boldsymbol{A}}&0&0\\
			0&\hat{\Sigma}_{\boldsymbol{A}}\epsilon&0\\
			0&0&0
		\end{pmatrix}
		\begin{pmatrix}
			\widetilde{\boldsymbol{W}}_{1}^{*} \\
			\widetilde{\boldsymbol{W}}_{2}^{*}\\
			\widetilde{\boldsymbol{W}}_{3}^{*}
		\end{pmatrix}
		=\widetilde{\boldsymbol{U}}
		\begin{pmatrix}
			I_{r_1}&0&0\\
			0&I_{r_2}\epsilon&0\\
			0&0&0
		\end{pmatrix}
		\begin{blockarray}{cc}
			n&  \\
			\begin{block}{(c)c}
				{\sf \Sigma}_{\boldsymbol{A}}\widetilde{\boldsymbol{W}}_{1}^{*} & r_1 \\
				\hat{\Sigma}_{\boldsymbol{A}}\widetilde{\boldsymbol{W}}_{2}^{*}  & r_2 \\
				\widetilde{\boldsymbol{W}}_{3}^{*}  & n-r_1-r_2  \\
			\end{block}
		\end{blockarray}.
	\end{equation}
	
	Let 
	\begin{equation*}
		\boldsymbol{M}=\begin{pmatrix}
			I_{r_{1}}&0&0\\
			\begin{pmatrix}
				-\boldsymbol{B}_{2}\boldsymbol{W}_{11}({\sf \Sigma}_{\boldsymbol{B}}^{1})^{-1}&0&0
			\end{pmatrix}\hat{\boldsymbol{T}}^{*}&I_{r_{2}}&0\\
			0&0&I_{n-r_{1}-r_{2}}
		\end{pmatrix},  \  \
		{\sf D}_{\boldsymbol{A}}=
		\begin{pmatrix}
			I_{r_1}&0&0\\
			0&I_{r_2}\epsilon&0\\
			0&0&0
		\end{pmatrix}.	
	\end{equation*} 
	Notice that the $(2, 1)$-block dual quaternion submatrix of $\boldsymbol{M}$ is underdetermined. In order to make the main proof smooth, we postpone the existence illustration of $\boldsymbol{M}$ to the end of the proof.  Obviously, $\boldsymbol{M}$ is a blocked elementary dual quaternion matrix, so it is nonsingular. Implanting $\boldsymbol{M}$ and its inverse in the right-hand side matrix of \eqref{DSVDA} gives
	\begin{equation}\label{AUX}
		\boldsymbol{A}=\widetilde{\boldsymbol{U}}{\sf D}_{\boldsymbol{A}}\boldsymbol{M}\boldsymbol{M}^{-1}\begin{pmatrix}
			{\sf \Sigma}_{\boldsymbol{A}}\widetilde{\boldsymbol{W}}_{1}^{*}\\
			\hat{\Sigma}_{\boldsymbol{A}}\widetilde{\boldsymbol{W}}_{2}^{*}\\
			\widetilde{\boldsymbol{W}}_{3}^{*}
		\end{pmatrix}=\widetilde{\boldsymbol{U}}\begin{pmatrix}
			I_{r_{1}}&0&0\\
			\begin{pmatrix}
				-\boldsymbol{B}_{2}\boldsymbol{W}_{11}({\sf \Sigma}_{\boldsymbol{B}}^{1})^{-1}&0&0
			\end{pmatrix}\hat{\boldsymbol{T}}^{*}\epsilon&I_{r_{2}}\epsilon&0\\
			0&0&0
		\end{pmatrix}\boldsymbol{X}_{1}^{-1},
	\end{equation}
	where
	\begin{equation}\label{X}
		\begin{split}
			\boldsymbol{X}_{1}^{-1}=\boldsymbol{M}^{-1}\begin{pmatrix}
				{\sf \Sigma}_{\boldsymbol{A}}\widetilde{\boldsymbol{W}}_{1}^{*}\\
				\hat{\Sigma}_{\boldsymbol{A}}\widetilde{\boldsymbol{W}}_{2}^{*}\\
				\widetilde{\boldsymbol{W}}_{3}^{*}
			\end{pmatrix}.	
		\end{split}
	\end{equation} 
	
	It is turn to consider the decomposition of the  $n\times p$ dual quaternion matrix 
	\begin{equation}\label{MinvB}
		\boldsymbol{X}_{1}^{-1}\boldsymbol{B}=\boldsymbol{M}^{-1}\begin{pmatrix}
			{\sf \Sigma}_{\boldsymbol{A}}\widetilde{\boldsymbol{W}}_{1}^{*}\\
			\hat{\Sigma}_{\boldsymbol{A}}\widetilde{\boldsymbol{W}}_{2}^{*}\\
			\widetilde{\boldsymbol{W}}_{3}^{*}
		\end{pmatrix}\boldsymbol{B}.
	\end{equation} 
	To this end, partition $\boldsymbol{X}_{1}^{-1}\boldsymbol{B}$ to be a 3-by-1 blocked matrix and applying \cref{PrePSVD}, we have	\begin{equation}\label{XinvB}
		\boldsymbol{X}_{1}^{-1}\boldsymbol{B}\equiv
		\begin{pmatrix}
			\boldsymbol{B}_{1}\\
			\boldsymbol{B}_{2}\\
			\boldsymbol{B}_{3}
		\end{pmatrix}=\boldsymbol{T}{\boldsymbol D}_{\boldsymbol{B}}\boldsymbol{Y},
	\end{equation} 
	where $\boldsymbol{T}\in {\bf \mathbb{DQ}}^{n \times n}$, ${\boldsymbol D}_{\boldsymbol{B}}\in {\bf \mathbb{DQ}}^{n \times p}$ are given by \eqref{PPSVD-T} and \eqref{PPSVD-B}, respectively, and the submatrices 
	$\hat{\boldsymbol{T}}=(\boldsymbol{T}_{11},\boldsymbol{T}_{12},\boldsymbol{T}_{13})\in {\bf \mathbb{DQ}}^{r_{1} \times r_{1}},$
	$\boldsymbol{U}_{1}\in {\bf \mathbb{DQ}}^{r_{2} \times r_{2}},$
	$\boldsymbol{U}_{2}\in {\bf \mathbb{DQ}}^{(n-r_1-r_2) \times (n-r_1-r_2)}$ are all unitary matrices, 
	$\boldsymbol{G} \in {\bf \mathbb{DQ}}^{(p-r_{11})\times (p-r_{11})}$, $\boldsymbol{Y}\in {\bf \mathbb{DQ}}^{p \times p}$ are both nonsingular matrices given by \eqref{LEMV}.
	
	The relation \eqref{XinvB} provides us an alternative representation of  $\boldsymbol{B}$ to be
	\begin{equation*}
		\boldsymbol{B}=\boldsymbol{X}_{1}\boldsymbol{T}{\boldsymbol D}_{\boldsymbol{B}}\boldsymbol{Y}=\boldsymbol{X}{\boldsymbol D}_{\boldsymbol{B}}\boldsymbol{Y}
	\end{equation*}
	with  
	\begin{equation}\label{X1T}
		\boldsymbol{X}=\boldsymbol{X}_{1}\boldsymbol{T}\in{\bf \mathbb{DQ}}^{n \times n}.
	\end{equation}
	
	The next object is to find an $m\times m$ unitary dual quaternion matrix $\boldsymbol{U}$ such that
	\begin{equation*}
		\boldsymbol{A}=\boldsymbol{U}{\sf D}_{\boldsymbol{A}}\boldsymbol{T}^{-1}\boldsymbol{X}_{1}^{-1}=\boldsymbol{U}{\sf D}_{\boldsymbol{A}}\boldsymbol{X}^{-1}.
	\end{equation*}
	
	Let \begin{equation*}
		\boldsymbol{P}=\begin{pmatrix}
			\hat{\boldsymbol{T}}&0&0\\
			\boldsymbol{U}_{1}\begin{pmatrix}
				\boldsymbol{B}_{2}\boldsymbol{W}_{11}(\sf {\Sigma}_{\boldsymbol{B}}^{1})^{-1}&0&0
			\end{pmatrix}\hat{\boldsymbol{T}}^{*}\epsilon&\boldsymbol{U}_{1}&0\\
			0&0&I_{m-r_1-r_2}
		\end{pmatrix}.
	\end{equation*}
	Then the block elementary transformations yield
	\begin{equation}\label{PDAM}
		\boldsymbol{P}{\sf D}_{\boldsymbol{A}}\boldsymbol{M}=\boldsymbol{P}\begin{pmatrix}
			I_{r_{1}}&0&0\\
			0&I_{r_{2}}\epsilon&0\\
			0&0&0
		\end{pmatrix}\boldsymbol{M}=\begin{pmatrix}
			I_{r_{1}}&0&0\\
			0&I_{r_{2}}\epsilon&0\\
			0&0&0
		\end{pmatrix}\boldsymbol{M}\boldsymbol{T}=\begin{pmatrix}
			\hat{\boldsymbol{T}}&0&0\\
			0&\boldsymbol{U}_{1}\epsilon&0\\
			0&0&0
		\end{pmatrix},
	\end{equation}
	which implies ${\sf D}_{\boldsymbol{A}}\boldsymbol{M}=\boldsymbol{P}{\sf D}_{\boldsymbol{A}}\boldsymbol{M}\boldsymbol{T}^{-1}$. Then,  recalling \eqref{AUX}-\eqref{X} and \eqref{X1T} and repeated application of \eqref{PDAM} gives
	\begin{equation*}
		\begin{split}
			\boldsymbol{A}&=\widetilde{\boldsymbol{U}}{\sf D}_{\boldsymbol{A}}\boldsymbol{M}\boldsymbol{X}_{1}^{-1}
			=\widetilde{\boldsymbol{U}}\boldsymbol{P}{\sf D}_{\boldsymbol{A}}\boldsymbol{M}\boldsymbol{T}^{-1}	\boldsymbol{X}_{1}^{-1}\\ &=\widetilde{\boldsymbol{U}}\begin{pmatrix}
				\hat{\boldsymbol{T}}&0&0\\
				0&\boldsymbol{U}_{1}\epsilon&0\\
				0&0&0
			\end{pmatrix}\boldsymbol{T}^{-1}\boldsymbol{X}_{1}^{-1} 
			=\widetilde{\boldsymbol{U}}\begin{pmatrix}
				\hat{\boldsymbol{T}}&0&0\\
				0&\boldsymbol{U}_{1}&0\\
				0&0&I_{m-r_1-r_2}
			\end{pmatrix}\begin{pmatrix}
				I_{r_1}&0&0\\
				0&I_{r_2}\epsilon&0\\
				0&0&0
			\end{pmatrix}\boldsymbol{X}^{-1} \\
			&=\boldsymbol{U}{\sf D}_{\boldsymbol{A}}\boldsymbol{X}^{-1},
		\end{split}
	\end{equation*}
	where
	\begin{equation}\label{QU}
		\widetilde{\boldsymbol{Q}}=\begin{pmatrix}
			\hat{\boldsymbol{T}}&0&0\\
			0&\boldsymbol{U}_{1}&0\\
			0&0&I_{m-r_1-r_2}
		\end{pmatrix},\ \ \boldsymbol{U}=\widetilde{\boldsymbol{U}}\widetilde{\boldsymbol{Q}},
	\end{equation}
	are both $m\times m$ unitary dual quaternion matrices.
	Then the desired decomposition \eqref{pgsv} is obtained.
	
	The rest proof is concentrated on the existence of the blocked elementary dual quaternion matrix $\boldsymbol{M}$. As a matter of fact, it is enough to illustrate the existence of $\boldsymbol{B}_{2}$ in the $(2, 1)$-block submatrix of $\boldsymbol{M}$.	From \eqref{MinvB}-\eqref{XinvB} it is known that  $\boldsymbol{M}$ satisfies 
	\begin{equation}\label{existM}
		\boldsymbol{M}^{-1}\begin{pmatrix}
			{\sf \Sigma}_{\boldsymbol{A}}\widetilde{\boldsymbol{W}}_{1}^{*}\\
			\hat{\Sigma}_{\boldsymbol{A}}\widetilde{\boldsymbol{W}}_{2}^{*}\\
			\widetilde{\boldsymbol{W}}_{3}^{*}
		\end{pmatrix}\boldsymbol{B}\equiv\begin{pmatrix}
			\boldsymbol{B}_{1}\\
			\boldsymbol{B}_{2}\\
			\boldsymbol{B}_{3}
		\end{pmatrix},
	\end{equation}
	and in the remaining part we will illustrate the rationality of the existence of \eqref{existM}. 	
	The explicit form of \eqref{existM} is
	\begin{equation}\label{PSVDeq1}
		\begin{split}
			&\begin{pmatrix}
				I_{r_{1}}&0&0\\
				\begin{pmatrix}
					\boldsymbol{B}_{2}\boldsymbol{W}_{11}({\sf \Sigma}_{\boldsymbol{B}}^{1})^{-1}&0&0
				\end{pmatrix}\hat{\boldsymbol{T}}^{*}&I_{r_{2}}&0\\
				0&0&I_{n-r_{1}-r_{2}}
			\end{pmatrix}\begin{pmatrix}
				{\sf \Sigma}_{\boldsymbol{A}}\widetilde{\boldsymbol{W}}_{1}^{*}\\
				\hat{\Sigma}_{\boldsymbol{A}}\widetilde{\boldsymbol{W}}_{2}^{*}\\
				\widetilde{\boldsymbol{W}}_{3}^{*}
			\end{pmatrix}\boldsymbol{B}\\
			&=\begin{pmatrix}
				{\sf \Sigma}_{\boldsymbol{A}}\widetilde{\boldsymbol{W}}_{1}^{*}\boldsymbol{B}\\
				(\boldsymbol{B}_{2}\boldsymbol{W}_{11}({\sf \Sigma}_{\boldsymbol{B}}^{1})^{-1}\boldsymbol{T}_{11}^{*}{\sf \Sigma}_{\boldsymbol{A}}\widetilde{\boldsymbol{W}}_{1}^{*}+\hat{\Sigma}_{\boldsymbol{A}}\widetilde{\boldsymbol{W}}_{2}^{*})\boldsymbol{B}\\
				\widetilde{\boldsymbol{W}}_{3}^{*}\boldsymbol{B}
			\end{pmatrix}=\begin{pmatrix}
				\boldsymbol{B}_{1}\\
				\boldsymbol{B}_{2}\\
				\boldsymbol{B}_{3}
			\end{pmatrix}.
		\end{split}
	\end{equation}
	From \eqref{PSVDeq1} we have the following equations
	\begin{subequations}\label{blkB12}
		\begin{align}
			{\sf \Sigma}_{\boldsymbol{A}}\widetilde{\boldsymbol{W}}_{1}^{*}\boldsymbol{B}&=\boldsymbol{B}_{1},  \\
			(\boldsymbol{B}_{2}\boldsymbol{W}_{11}({\sf \Sigma}_{\boldsymbol{B}}^{1})^{-1}\boldsymbol{T}_{11}^{*}{\sf \Sigma}_{\boldsymbol{A}}\widetilde{\boldsymbol{W}}_{1}^{*}+\hat{\Sigma}_{\boldsymbol{A}}\widetilde{\boldsymbol{W}}_{2}^{*})\boldsymbol{B}&=\boldsymbol{B}_{2}, \\
			\widetilde{\boldsymbol{W}}_{3}^{*}\boldsymbol{B} &=\boldsymbol{B}_{3}.
		\end{align}
	\end{subequations}
	The equations \eqref{blkB12}(a) and \eqref{blkB12}(c) have illustrated the existence of $\boldsymbol{B}_{1}, \boldsymbol{B}_{3}$. We further discuss the existence of $\boldsymbol{B}_{2}$. A simple transformation of \eqref{blkB12}(b) gives 
	\begin{equation}\label{blkB2alt}
		\boldsymbol{B}_{2}(I_p-\boldsymbol{W}_{11}({\sf \Sigma}_{\boldsymbol{B}}^{1})^{-1}\boldsymbol{T}_{11}^{*}{\sf \Sigma}_{\boldsymbol{A}}\widetilde{\boldsymbol{W}}_{1}^{*}\boldsymbol{B})=\hat{\Sigma}_{\boldsymbol{A}}\widetilde{\boldsymbol{W}}_{2}^{*}\boldsymbol{B}.
	\end{equation}
	Substituting  \eqref{blkB12}(a) into \eqref{blkB2alt} we obtain
	\begin{equation}\label{B2B1}
		\boldsymbol{B}_{2}(I_p-\boldsymbol{W}_{11}({\sf \Sigma}_{\boldsymbol{B}}^{1})^{-1}\boldsymbol{T}_{11}^{*}\boldsymbol{B}_{1})=\hat{\Sigma}_{\boldsymbol{A}}\widetilde{\boldsymbol{W}}_{2}^{*}\boldsymbol{B}.
	\end{equation}
	
	From the DQSVD of $\boldsymbol{B}_{1}$ given in \eqref{B1SVD}, we have 
	\begin{equation}\label{eqB1eq}
		\boldsymbol{B}_{1}=\boldsymbol{T}_{11}{\sf \Sigma}_{\boldsymbol B}^{1}\boldsymbol{W}_{11}^{*}+\boldsymbol{T}_{12}\hat{\Sigma}_{\boldsymbol B}\boldsymbol{W}_{12}^{*}\epsilon.
	\end{equation}
	Then	substituting \eqref{eqB1eq} into \eqref{B2B1} gives
	\begin{equation}\label{eq:B2proj}
		\boldsymbol{B}_{2}(I_p-\boldsymbol{W}_{11}\boldsymbol{W}_{11}^{*})=\hat{\Sigma}_{\boldsymbol{A}}\widetilde{\boldsymbol{W}}_{2}^{*}\boldsymbol{B},
	\end{equation}
	It is obvious that there exists a dual quaternion matrix $\boldsymbol{B}_{2}$ such that \eqref{eq:B2proj} holds.
\end{proof}

Compared with the PSVD of real matrix pair \cite{Heath}, the dual quaternion factor matrix $\boldsymbol{Y}$ of $\boldsymbol{B}$ in \eqref{pgsv} is nonsingular rather than unitary. Even so, starting from the DQPSVD \eqref{pgsv} we can as well find DQSVD of the product of $\boldsymbol{A}$ and $\boldsymbol{B}$.  Explicitly, the product of $\boldsymbol{A}$ and $\boldsymbol{B}$ in \eqref{pgsv} gives a decomposition of $\boldsymbol{AB}$ as follows 
\begin{equation}\label{PSVDABpro}
	\begin{aligned}
		\boldsymbol{A}\boldsymbol{B}&=
		\boldsymbol{U}{\sf D}_{\boldsymbol{A}}{\boldsymbol D}_{\boldsymbol{B}}\boldsymbol{Y}\\
		&=\boldsymbol{U}\begin{blockarray}{ccc}
			r_{1} & r_{2} & n-(r_{1}+r_{2}) \\
			\begin{block}{(ccc)}
				I_{r_{1}}&0&0\\
				0&I_{r_{2}}\epsilon  &0\\
				0&0&0\\
			\end{block}
		\end{blockarray}\begin{blockarray}{ccc}
			r_{11}&p-r_{11}\\
			\begin{block}{(cc)c}
				{\sf \Sigma}_{\boldsymbol{B}}^{1}&0&r_{11}\\
				0&\begin{pmatrix}
					\hat{\Sigma}_{\boldsymbol{B}}\epsilon&0\\
					0&0
				\end{pmatrix}\boldsymbol{G}^{-1}&r_{1}-r_{11}\\
				0&\begin{pmatrix}
					{\boldsymbol \Sigma}_{\boldsymbol{B}}^{2}\\
					{\boldsymbol \Sigma}_{\boldsymbol{B}}^{3}
				\end{pmatrix}&n-r_1\\
			\end{block}
		\end{blockarray}\boldsymbol{Y}\\
		&=\boldsymbol{U}\begin{blockarray}{ccc}
			r_{11}&p-r_{11}\\
			\begin{block}{(cc)c}
				{\sf \Sigma}_{\boldsymbol{B}}^{1}&0&r_{11}\\
				0&\begin{pmatrix}
					\hat{\Sigma}_{\boldsymbol{B}}\epsilon&0\\
					0&0
				\end{pmatrix}&r_{1}-r_{11}\\
				0&\begin{pmatrix}
					{\boldsymbol \Sigma}_{\boldsymbol{B}}^{2}\boldsymbol{G}\epsilon\\
					0
				\end{pmatrix}&m-r_1\\
			\end{block}
		\end{blockarray}\boldsymbol{W}^{*},
	\end{aligned}
\end{equation}
where in the last equality we have used the relation \eqref{LEMV}.
Let 
\begin{equation*}
	{\bf F}\epsilon=\begin{pmatrix}
		\begin{pmatrix}
			\hat{\Sigma}_{\boldsymbol{B}}\epsilon&0\\
			0&0\\
		\end{pmatrix}\\
		{\boldsymbol \Sigma}_{\boldsymbol{B}}^{2}\boldsymbol{G}\epsilon
	\end{pmatrix}.
\end{equation*}
Then, from the QSVD of the quaternion matrix ${\bf F}$, there exist unitary quaternion matrices 
\begin{equation*}
	\hat{\bf H} \in \mathbb{Q}^{(r_{1}+r_{2}-r_{11}) \times (r_{1}+r_{2}-r_{11})},  \  \  
	\hat{\bf N} \in \mathbb{Q}^{(p-r_{11}) \times (p-r_{11})}
\end{equation*} 
such that
\begin{equation}\label{SVDF}
	{\bf F}=\hat{\bf H} \begin{pmatrix}\Sigma_{\boldsymbol{AB}} &  0 \\  0  &  0  \end{pmatrix}\hat{\bf N}^*
\end{equation}
with positive real diagonal matrix ${\Sigma_{\boldsymbol{AB}}}={\rm diag}({\nu}_{1},{\nu}_{2},\ldots,{\nu}_{\hat{r}_{11}}).$
Implanting \eqref{SVDF} into \eqref{PSVDABpro} we eventually obtain the DQSVD of the product dual quaternion matrices $\boldsymbol{A}\boldsymbol{B}$ to be
\begin{equation*}
	\begin{split}
		\boldsymbol{A}\boldsymbol{B}&=\boldsymbol{U}\begin{blockarray}{ccc}
			r_{11}&r_{1}+r_{2}-r_{11}&m-r_{1}-r_{2}\\
			\begin{block}{(ccc)}
				I&0&0\\
				0&\hat{{\bf H}}&0\\
				0&0&I\\
			\end{block}
		\end{blockarray}\begin{blockarray}{cccc}
			r_{11}&\hat{r}_{11}&p-r_{11}-\hat{r}_{11}\\
			\begin{block}{(ccc)c}
				{\sf \Sigma}_{\boldsymbol{B}}^{1}&0&0&r_{11}\\
				0&{\Sigma_{\boldsymbol{AB}}}\epsilon&0&\hat{r}_{11}\\
				0&0&0&m-r_{11}-\hat{r}_{11}\\
			\end{block}
		\end{blockarray}\\
		& \quad \cdot 
		\left(\begin{array}{cc}
			I_{r_{11}}&0\\
			0&\hat{{\bf N}}^*
		\end{array}\right) \boldsymbol{W}^{*}\\
		&=\boldsymbol{H}\begin{pmatrix}
			{\sf \Sigma}_{\boldsymbol{B}}^{1}&0&0\\
			0&{\Sigma_{\boldsymbol{AB}}}\epsilon&0\\
			0&0&0\\
		\end{pmatrix}\boldsymbol{N}^*.
	\end{split}
\end{equation*}
It is obvious that 
\begin{equation*}
	\begin{split}
		\boldsymbol{H}&=\boldsymbol{U}\begin{blockarray}{ccc}
			r_{11}&r_{1}+r_{2}-r_{11}&n-r_{1}-r_{2}\\
			\begin{block}{(ccc)}
				I&0&0\\
				0&\hat{{\bf H}}&0\\
				0&0&I\\
			\end{block}
		\end{blockarray}\in\mathbb{DQ}^{m \times m},  \\ 
		\boldsymbol{N}&=\boldsymbol{W}\begin{blockarray}{ccc}
			\begin{block}{(cc)c}
				I&0&r_{11}\\
				0&\hat{{\bf N}}&p-r_{11}\\
			\end{block}
		\end{blockarray}
		\in \mathbb{DQ}^{p \times p},	
	\end{split}
\end{equation*}
are both unitary dual quaternion matrices.

Golub and Zha \cite[Theorem 2.1]{GolubZha1994} derived CCD decomposition for the given real matrices $A$ and $B$ having the same number of rows. Motivated by the real CCD decomposition, we present the DQCCD decomposition theorem of a dual quaternion matrix pair, and provide a short remark about the obtained decomposition.

\begin{theorem}[DQCCD]
	Suppose $\boldsymbol{A}\in {\bf \mathbb{DQ}}^{m \times n}$, $\boldsymbol{B}\in {\bf \mathbb{DQ}}^{m \times l}$, ${\rm rank}(\boldsymbol{A})=p$, ${\rm rank}(\boldsymbol{B})=q$. There exist a unitary matrix $\boldsymbol{Q}\in {\bf \mathbb{DQ}}^{m \times m}$ and two dual quaternion matrices $\boldsymbol{X_{\boldsymbol{A}}}\in {\bf \mathbb{DQ}}^{n \times n}$, $\boldsymbol{X_{\boldsymbol{B}}}\in {\bf \mathbb{DQ}}^{l \times l}$, such that\\
	$$\boldsymbol{A}=\boldsymbol{Q}({\sf \Sigma}_{\boldsymbol{A}},0)\boldsymbol{X_{\boldsymbol{A}}},\ \boldsymbol{B}=\boldsymbol{Q}(\Sigma_{\boldsymbol{B}},0)\boldsymbol{X_{\boldsymbol{B}}},$$
	where
	\begin{equation}\label{CCDSigm}
		{\sf \Sigma}_{\boldsymbol{A}}=\left( \begin{array}{ccccc}
			I&0&0&0&0\\
			0&I&0&0&0\\
			0&0&{\sf C}&0&0\\
			0&0&0&D\epsilon&0\\
			0&0&0&0&0\\
			\hline
			{\Sigma}\epsilon&0&0&0&0\\
			0&0&0&0&0\\
			0&0&{\sf S}&0&0\\
			0&0&0&I&0\\
			0&0&0&0&I
		\end{array}\right),\ 
		\Sigma_{\boldsymbol{B}}=\left( \begin{array}{ccc}
			I_{q}\\
			\hline
			0
		\end{array}\right),
	\end{equation}
	$\boldsymbol{Q}$, $\boldsymbol{X_{\boldsymbol{A}}}$, $\boldsymbol{X_{\boldsymbol{B}}}$ are given in	
	\eqref{QUV} and \eqref{Xab}, respectively. $\Sigma$, $D$ are real diagonal matrices with positive diagonal elements, $\sf C$ and $\sf S$ are appreciable dual matrix and
	\begin{align*}
		\Sigma&={\rm diag}{(\sigma_{1},\sigma_{2},\ldots,\sigma_{r})},\quad \sigma_{1}\ge\sigma_{2}\ge\cdots\ge\sigma_{r}>0,\\
		{\sf C}&={\rm diag}({\sf c}_{i+1},{\sf c}_{i+2},\ldots,{\sf c}_{i+j}),\ 1>{\sf c}_{i+1}\ge {\sf c}_{i+2}\ge\cdots\ge {\sf c}_{i+j}>0,\\
		{\sf S}&={\rm diag}({\sf s}_{i+1},{\sf s}_{i+2},\ldots,{\sf s}_{i+j}),\ 0<{\sf s}_{i+1}\le {\sf s}_{i+2}\le\cdots\le {\sf s}_{i+j}<1,\\
		& \qquad\qquad\qquad\qquad\qquad   {\sf C}^2+{\sf S}^2=I_j.
	\end{align*}
\end{theorem}
\begin{proof}
	From \cref{UD}  we  have the unitary decompositions of dual quaternion matrices $\boldsymbol{A}$ and $\boldsymbol{B}$  as follows
	\begin{subequations}\label{ABdeco}
		\begin{align}
			\boldsymbol{A} =\boldsymbol{Q}_{\boldsymbol{A}}\boldsymbol{R}_{1}
			&=(\boldsymbol{Q}_{\boldsymbol{A}},0)
			\begin{blockarray}{cc}
				n &    \\
				\begin{block}{(c)c}
					\boldsymbol{R}_{1} & p\\
					\hat{\boldsymbol{R}}_{1} &  n-p\\
				\end{block}
			\end{blockarray}  \equiv	(\boldsymbol{Q}_{\boldsymbol{A}},0)\boldsymbol{R}_{\boldsymbol{A}}, \\
			\boldsymbol{B} =\boldsymbol{Q}_{\boldsymbol{B}}\boldsymbol{R}_{2}
			&=(\boldsymbol{Q}_{\boldsymbol{B}},0)
			\begin{blockarray}{cc}
				l &    \\
				\begin{block}{(c)c}
					\boldsymbol{R}_{2}& q\\
					\hat{\boldsymbol{R}}_{2}& l-q\\
				\end{block}
			\end{blockarray} \equiv	(\boldsymbol{Q}_{\boldsymbol{B}},0)\boldsymbol{R}_{\boldsymbol{B}}.
		\end{align}
	\end{subequations}
	The choices of $\hat{\boldsymbol{R}}_{1}$ and $\hat{\boldsymbol{R}}_{2}$ are contingent on the specific forms of $\boldsymbol{R}_{1}$ and $\boldsymbol{R}_{2}$. Based on the distinct configurations of $\boldsymbol{R}_{1}$ and $\boldsymbol{R}_{2}$, three scenarios can be delineated as follows:
	\begin{itemize}
		\item
		The dual quaternion  row vectors of both $\boldsymbol{R}_{1}$ and $\boldsymbol{R}_{2}$ are appreciable.
		\item
		Either $\boldsymbol{R}_{1}$ or $\boldsymbol{R}_{2}$ contains infinitesimal dual quaternion row vectors.
		\item
		The nonzero elements of both $\boldsymbol{R}_{1}$ and $\boldsymbol{R}_{2}$ are infinitesimal dual quaternions, which implies that both $\boldsymbol{A}$ and $\boldsymbol{B}$ are infinitesimal dual quaternion matrices.
	\end{itemize}
	
	For the first case, all the dual quaternion row vectors of $\boldsymbol{R}_{1}$ and $\boldsymbol{R}_{2}$ are appreciable. This implies that ${\rm rank}(\boldsymbol{A})={\rm Arank}(\boldsymbol{A})=p$ and ${\rm rank}(\boldsymbol{B})={\rm Arank}(\boldsymbol{B})=q$. It is easy to find dual quaternion matrices $\hat{\boldsymbol{R}}_{1}$ and $\hat{\boldsymbol{R}}_{2}$ such that $\boldsymbol{R}_{\boldsymbol{A}}$ and $\boldsymbol{R}_{\boldsymbol{B}}$ are nonsingular, respectively. Expand $\boldsymbol{Q}_{\boldsymbol{A}}$ and $\boldsymbol{Q}_{\boldsymbol{B}}$ to get $\boldsymbol{Q}_{1}=(\boldsymbol{Q}_{\boldsymbol{A}},\hat{\boldsymbol{Q}}_{\boldsymbol{A}})$ and $\boldsymbol{Q}_{2}=(\boldsymbol{Q}_{\boldsymbol{B}},\hat{\boldsymbol{Q}}_{\boldsymbol{B}})$ such that they are both $m\times m$ unitary dual quaternion matrices. Therefore, $\boldsymbol{Q}_{2}^{*}\boldsymbol{Q}_{1}$ is also a unitary matrix. From \cref{coroCS}, there exist unitary matrices $\boldsymbol{U}_{1}\in {\bf \mathbb{DQ}}^{q\times q}$, $\boldsymbol{U}_{2}\in {\bf \mathbb{DQ}}^{(m-q)\times (m-q)}$ and $\boldsymbol{V}_{1}\in {\bf \mathbb{DQ}}^{p\times p}$ such that 
	\begin{equation}\label{QBACS}
		\begin{pmatrix}
			\boldsymbol{U}_{1}^{*}&0\\
			0&\boldsymbol{U}_{2}^{*}
		\end{pmatrix}\begin{pmatrix}
			\boldsymbol{Q}_{\boldsymbol{B}}^{*}\boldsymbol{Q}_{\boldsymbol{A}}\\
			\hat{\boldsymbol{Q}}_{\boldsymbol{B}}^{*}\boldsymbol{Q}_{\boldsymbol{A}}
		\end{pmatrix}\boldsymbol{V}_{1}={\sf \Sigma}_{\boldsymbol {A}},
	\end{equation}
	where ${\sf \Sigma}_{\boldsymbol {A}}$ is given by \eqref{CCDSigm}.
	
	Let 
	\begin{equation}\label{QUV}
		\begin{split}
			\boldsymbol{U}&={\rm diag}(\boldsymbol{U}_{1},\ \boldsymbol{U}_{2})\in {\bf \mathbb{DQ}}^{m\times m}, \  \ \widetilde{\boldsymbol{U}}={\rm diag}(\boldsymbol{U}_{1},\ I_{l-q})\in {\bf \mathbb{DQ}}^{l\times l},\\
			\widetilde{\boldsymbol{V}}&={\rm diag}(\boldsymbol{V}_{1},\ I_{n-p})\in {\bf \mathbb{DQ}}^{n\times n},\quad\boldsymbol{Q}=\boldsymbol{Q}_{2}\boldsymbol{U}\in {\bf \mathbb{DQ}}^{m\times m}.
		\end{split}
	\end{equation}
	From \eqref{QBACS} it is straightforward to verify the following relations:
	\begin{subequations}\label{QABSig}
		\begin{align}
			\boldsymbol{Q}_{\boldsymbol{A}} &=\boldsymbol{Q}_{2}\boldsymbol{U}{\sf \Sigma}_{\boldsymbol{A}}\boldsymbol{V}_{1}^{*}=\boldsymbol{Q}{\sf \Sigma}_{\boldsymbol{A}}\boldsymbol{V}_{1}^{*},  \\
			\boldsymbol{Q}_{\boldsymbol{B}}&=\boldsymbol{Q}_2\begin{pmatrix}
				I_{q}\\
				0
			\end{pmatrix}=\boldsymbol{Q}\boldsymbol{U}^{*}\begin{pmatrix}
				I_{q}\\
				0
			\end{pmatrix}=\boldsymbol{Q}\begin{pmatrix}
				\boldsymbol{U}_{1}^{*}\\
				0
			\end{pmatrix}=\boldsymbol{Q}\Sigma_{\boldsymbol{B}}\boldsymbol{U}_{1}^{*}.
		\end{align}
	\end{subequations}
	Substituting \eqref{QABSig} into \eqref{ABdeco} yields	
	\begin{equation*}
		\begin{aligned}
			\boldsymbol{A}&=(\boldsymbol{Q}_{\boldsymbol{A}},0)\boldsymbol{R}_{\boldsymbol{A}}=(\boldsymbol{Q}{\sf \Sigma}_{\boldsymbol{A}}\boldsymbol{V}_{1}^{*},0)\boldsymbol{R}_{\boldsymbol{A}}=\boldsymbol{Q}({\sf \Sigma}_{\boldsymbol{A}},0)\boldsymbol{X}_{\boldsymbol{A}},\\
			\boldsymbol{B}&=(\boldsymbol{Q}_{\boldsymbol{B}},0)\boldsymbol{R}_{\boldsymbol{B}}=(\boldsymbol{Q}\Sigma_{\boldsymbol{B}}\boldsymbol{U}_{1}^{*},0)\boldsymbol{R}_{\boldsymbol{B}}=\boldsymbol{Q}(\Sigma_{\boldsymbol{B}},0)\boldsymbol{X}_{\boldsymbol{B}},	\end{aligned}
	\end{equation*}
	where 
	\begin{equation}\label{Xab}
		\boldsymbol{X}_{\boldsymbol{A}}=\widetilde{\boldsymbol{V}}^{*}\boldsymbol{R}_{\boldsymbol{A}}\in {\bf \mathbb{DQ}}^{n\times n}, \quad \boldsymbol{X}_{\boldsymbol{B}}=\widetilde{\boldsymbol{U}}^{*}\boldsymbol{R}_{\boldsymbol{B}}\in {\bf \mathbb{DQ}}^{l\times l}
	\end{equation}
	are both nonsingular dual quaternion matrices. 
	
	Notice that the results in the first case can degenerate to those of the CCD on quaternion ring by taking the infinitesimal parts of both $\boldsymbol{A}$ and $\boldsymbol{B}$ to be zero matrices.	
	For the second case, if either $\boldsymbol{R}_{1}$ or $\boldsymbol{R}_{2}$ contains infinitesimal dual quaternion row vectors, then either $\boldsymbol{X}_{\boldsymbol{A}}$ or $\boldsymbol{X}_{\boldsymbol{B}}$ in \eqref{Xab} is singular.
	For the last case, if the nonzero elements of both $\boldsymbol{R}_{1}$ and $\boldsymbol{R}_{2}$ are infinitesimal dual quaternions, the problem becomes the CCD of quaternion matrix pair over quaternion ring except for the extra dual unit $\epsilon$.
\end{proof}

\begin{remark}
	Unlike the CCD of real matrix pair, after performing the unitary decompositions on the dual quaternion matrices $\boldsymbol{A}$ and $\boldsymbol{B}$ to obtain the trapezoidal dual quaternion matrices $\boldsymbol{R}_{1}$ and $\boldsymbol{R}_{2}$, it is not ensured that these matrices can be extended into a nonsingular matrix. This is because some of dual quaternion row vectors might be  infinitesimal in $\boldsymbol{R}_{1}$ and/or $\boldsymbol{R}_{2}$. As a matter of fact, if either $\boldsymbol{R}_{1}$ or $\boldsymbol{R}_{2}$ contains infinitesimal dual quaternion row vectors, the derived findings have forfeited the inherent significance of canonical correlation among the columns of $\boldsymbol{A}$ and $\boldsymbol{B}$.  
\end{remark}

\section{Example illustration}\label{examp5}

In this section, we present three artificially toy examples to verify the principle of different quotient-type SVD of the dual quaternion matrix pair $\{{\boldsymbol A}, {\boldsymbol B}\}$, under the assumption ${\boldsymbol A}$ and ${\boldsymbol B}$ having the same number of columns. Example verifications for the principle of DQPSVD and DQCCD are left to the reader.


\begin{example}\label{examp1}
	Find the quotient-type singular value decomposition of the dual quaternion matrix pair $\{{\boldsymbol A}, {\boldsymbol B}\}$, where
	\begin{equation*}
		\boldsymbol{A}= \begin{pmatrix}
			\frac{\sqrt{2}}{2}+\epsilon&\frac{1}{2}{\bf k}\epsilon&0&0\\
			\frac{\sqrt{2}}{2}{\bf k}\epsilon&\frac{1}{2}&0&0\\
			0&0&0&0
		\end{pmatrix},\quad
		\boldsymbol{B}=\begin{pmatrix}
			\frac{\sqrt{2}}{2}\epsilon&\frac{1}{2}{\bf j}\epsilon&0&0\\
			0&\frac{1}{2}&0&0\\
			0&0&\frac{\sqrt{2}}{2}\epsilon&0
		\end{pmatrix}.
	\end{equation*}
\end{example}

{\it Solution.}
First of all, according to \cref{lm:4.1} there exist two unitary dual quaternion matrices 
\begin{equation*}
	\boldsymbol{P}=\begin{pmatrix}
		1&(\frac{\sqrt{2}}{2}{\bf k}-1)\epsilon&0&0&-\epsilon&-\frac{\sqrt{2}}{2}{\bf k}\epsilon\\
		({\bf k}+\frac{\sqrt{2}}{2})\epsilon&\frac{\sqrt{2}}{2}&0&0&0&-\frac{\sqrt{2}}{2}\\
		0&0&\epsilon&1&0&0\\
		\epsilon&\frac{\sqrt{2}}{2}{\bf j}\epsilon&0&0&1&\frac{\sqrt{2}}{2}{\bf j}\epsilon\\
		\frac{\sqrt{2}}{2}\epsilon&\frac{\sqrt{2}}{2}&0&0&{\bf j}\epsilon&\frac{\sqrt{2}}{2}\\
		0&0&1&-\epsilon&0&0
	\end{pmatrix}
\end{equation*} and \begin{equation*}
	\boldsymbol{Q}=\begin{pmatrix}
		1&-\epsilon&0&0\\
		\epsilon&1&0&0\\
		0&0&1&-\epsilon\\
		0&0&\epsilon&1\\
	\end{pmatrix},
\end{equation*} 
such that the DQSVD of $\boldsymbol{C}=\begin{pmatrix}
	\boldsymbol{A}\\
	\boldsymbol{B}
\end{pmatrix}$ is
\begin{equation*}
	\boldsymbol{P}^{*}\boldsymbol{C}\boldsymbol{Q}
	=\left(\begin{array}{c|c}
		{\sf \Sigma}_{\boldsymbol{C}}&0  \\
		\hline
		0&0  \end{array}\right)
	\equiv
	\left(\begin{array}{cc|c}
		{\sf \Sigma}_{t}&0&0  \\
		0&{\Sigma}_{s}\epsilon&0   \\
		\hline
		0&0&0  \\
	\end{array}\right)
	=\left(\begin{array}{ccc|c}
		\frac{\sqrt{2}}{2}+\epsilon&0&0&0\\
		0&\frac{\sqrt{2}}{2}&0&0\\
		0&0&\frac{\sqrt{2}}{2}\epsilon&0\\
		\hline
		0&0&0&0\\
		0&0&0&0\\
		0&0&0&0
	\end{array}\right),
\end{equation*}
from which we have  ${\sf \Sigma}_{\boldsymbol{C}} ={\rm diag}({\sf \Sigma}_{t},{\Sigma}_{s}\epsilon)$ and
\begin{equation*}
	{\sf \Sigma}_{t}={\rm diag}(\frac{\sqrt{2}}{2}+\epsilon,\frac{\sqrt{2}}{2}),\quad \Sigma_{s}=\frac{\sqrt{2}}{2}.
\end{equation*}
Here, we have used \eqref{SVDP} with $s=1,t=2$. 
Partition $\boldsymbol{P}$ as the following 2-by-2 blocked matrix
\begin{equation*}
	\boldsymbol{P}=\begin{pmatrix}
		\boldsymbol{P}_{11}&\boldsymbol{P}_{12}\\
		\boldsymbol{P}_{21}&\boldsymbol{P}_{22}
	\end{pmatrix}=
	\left(\begin{array}{ccc|ccc}
		1&(\frac{\sqrt{2}}{2}{\bf k}-1)\epsilon&0&0&-\epsilon&-\frac{\sqrt{2}}{2}{\bf k}\epsilon\\
		({\bf k}+\frac{\sqrt{2}}{2})\epsilon&\frac{\sqrt{2}}{2}&0&0&0&-\frac{\sqrt{2}}{2}\\
		0&0&\epsilon&1&0&0\\
		\hline
		\epsilon&\frac{\sqrt{2}}{2}{\bf j}\epsilon&0&0&1&\frac{\sqrt{2}}{2}{\bf j}\epsilon\\
		\frac{\sqrt{2}}{2}\epsilon&\frac{\sqrt{2}}{2}&0&0&{\bf j}\epsilon&\frac{\sqrt{2}}{2}\\
		0&0&1&-\epsilon&0&0\\
	\end{array}\right).
\end{equation*}
From \cref{coroCS}  the DQCS decomposition of $\begin{pmatrix}
	\boldsymbol{P}_{11}\\
	\boldsymbol{P}_{21}
\end{pmatrix}$ 
is derived as follows
\begin{equation}\label{PPGSVD1}
	\begin{aligned}	
		\begin{pmatrix}
			\boldsymbol{P}_{11}\\
			\boldsymbol{P}_{21}
		\end{pmatrix}
		&=\begin{pmatrix}
			\boldsymbol{U}& 0  \\
			0  &\boldsymbol{V}
		\end{pmatrix}\begin{pmatrix}
			{\sf \Sigma}_{\boldsymbol{A}}\\
			{\sf \Sigma}_{\boldsymbol{B}}
		\end{pmatrix}\boldsymbol{W}^{*}\\
		&=\left(\begin{array}{ccc|ccc}
			1&{\bf k}\epsilon&0&0&0&0\\
			{\bf k}\epsilon&1&0&0&0&0\\
			0&0&1&0&0&0\\
			\hline
			0&0&0&1&{\bf j}\epsilon&0\\
			0&0&0&{\bf j}\epsilon&1&0\\
			0&0&0&0&0&1
		\end{array}\right)
		\begin{pmatrix}
			1&0&0\\
			0&\frac{\sqrt{2}}{2}&0\\
			0&0&\epsilon\\
			\hline
			\epsilon&0&0\\
			0&\frac{\sqrt{2}}{2}&0\\
			0&0&1
		\end{pmatrix}\begin{pmatrix}
			1&-\epsilon&0\\
			\epsilon&1&0\\
			0&0&1
		\end{pmatrix}.
	\end{aligned}
\end{equation}
Then, from \eqref{QSVDX1} we obtain a singular dual quaternion matrix
\begin{equation}\label{XGSVD1}
	\boldsymbol{X}={\rm diag}({\boldsymbol{W}}^*{\sf \Sigma}_{\boldsymbol{C}},1)\boldsymbol{Q}^{*}
	=\begin{pmatrix}
		\frac{\sqrt{2}}{2}+\epsilon&0&0&0\\
		0&\frac{\sqrt{2}}{2}&0&0\\
		0&0&\frac{\sqrt{2}}{2}\epsilon&0\\
		0&0&-\epsilon&1\\
	\end{pmatrix}.
\end{equation} 

Combining \eqref{XGSVD1} with the submatrices in \eqref{PPGSVD1} we obtain the DQGSVD1 of $\{\boldsymbol{A}, \boldsymbol{B}\}$
in the form of \eqref{eq:AB1}, that is,
\begin{equation*}
	\begin{aligned}
		\boldsymbol{A}&=\boldsymbol{U}({\sf \Sigma}_{\boldsymbol{A}}, \  0)\boldsymbol{X}
		=\left(\begin{array}{ccc}
			1&{\bf k}\epsilon&0\\
			{\bf k}\epsilon&1&0\\
			0&0&1
		\end{array}\right)
		\left(\begin{array}{ccc|c}
			1&0&0& 0\\
			0&\frac{\sqrt{2}}{2}&0&0\\
			0&0&\epsilon&0
		\end{array}\right)\begin{pmatrix}
			\frac{\sqrt{2}}{2}+\epsilon&0&0&0\\
			0&\frac{\sqrt{2}}{2}&0&0\\
			0&0&\frac{\sqrt{2}}{2}\epsilon&0\\
			0&0&-\epsilon&1\\
		\end{pmatrix},   
		\\
		\boldsymbol{B}&=\boldsymbol{V}({\sf \Sigma}_{\boldsymbol{B}}, \ 0)\boldsymbol{X}
		=\left(\begin{array}{ccc}
			1&{\bf j}\epsilon&0\\
			{\bf j}\epsilon&1&0\\
			0&0&1
		\end{array}\right)
		\left(\begin{array}{ccc|c}
			\epsilon&0&0&0\\
			0&\frac{\sqrt{2}}{2}&0&0\\
			0&0&1&0
		\end{array}\right)\begin{pmatrix}
			\frac{\sqrt{2}}{2}+\epsilon&0&0&0\\
			0&\frac{\sqrt{2}}{2}&0&0\\
			0&0&\frac{\sqrt{2}}{2}\epsilon&0\\
			0&0&-\epsilon&1\\
		\end{pmatrix}.
	\end{aligned}
\end{equation*}


In the next example we demonstrate the other type of the quotient-type singular value decomposition of $\{{\boldsymbol A}, {\boldsymbol B}\}$, which contain a nonsingular dual quaternion matrix ${\boldsymbol X}$. 

\begin{example}\label{examp2}
	Find the quotient-type singular value decomposition of the dual quaternion matrix pair $\{{\boldsymbol A}, {\boldsymbol B}\}$, where
	\begin{equation}\label{ex2AB2}
		\boldsymbol{A}=\begin{pmatrix}
			\frac{\sqrt{2}}{2}+\epsilon&\frac{1}{2}{\bf k}\epsilon&0&0&0\\
			\frac{\sqrt{2}}{2}{\bf k}\epsilon&\frac{1}{2}&0&0&0\\
			0&0&\frac{\sqrt{2}}{2}\epsilon&\epsilon&0
		\end{pmatrix}, \quad
		\boldsymbol{B}=\begin{pmatrix}
			\frac{\sqrt{2}}{2}\epsilon&\frac{1}{2}{\bf j}\epsilon&0&0&\epsilon\\
			0&\frac{1}{2}&0&0&0\\
			0&0&\frac{\sqrt{2}}{2}-\epsilon&\frac{\sqrt{2}}{2}\epsilon&0
		\end{pmatrix}.
	\end{equation}
\end{example}

{\it Solution.}
Similar to Example \ref{examp1}, according to \cref{lm:4.1} there exist two unitary dual quaternion matrices 
\begin{equation*}
	\boldsymbol{P}=\begin{pmatrix}
		1&(\frac{\sqrt{2}}{2}{\bf k}-1)\epsilon&0&0&-\epsilon&-\frac{\sqrt{2}}{2}{\bf k}\epsilon\\
		({\bf k}+\frac{\sqrt{2}}{2})\epsilon&\frac{\sqrt{2}}{2}&0&0&0&-\frac{\sqrt{2}}{2}\\
		0&0&\epsilon&1&0&0\\
		\epsilon&\frac{\sqrt{2}}{2}{\bf j}\epsilon&0&0&1&\frac{\sqrt{2}}{2}{\bf j}\epsilon\\
		\frac{\sqrt{2}}{2}\epsilon&\frac{\sqrt{2}}{2}&0&0&{\bf j}\epsilon&\frac{\sqrt{2}}{2}\\
		0&0&1&-\epsilon&0&0
	\end{pmatrix}
\end{equation*} and 
\begin{equation}\label{Ex2Q}
	\boldsymbol{Q}=\begin{pmatrix}
		1&-\epsilon&0&0&0\\
		\epsilon&1&0&0&0\\
		0&0&1&-\epsilon&0\\
		0&0&\epsilon&1&0\\
		0&0&0&0&1
	\end{pmatrix},
\end{equation} 
such that the DQSVD of $\boldsymbol{C}=\begin{pmatrix}
	\boldsymbol{A}\\
	\boldsymbol{B}
\end{pmatrix}$ is
\begin{equation}\label{Ex1DQsvd}
	\boldsymbol{P}^{*}\boldsymbol{C}\boldsymbol{Q}
	=\left(\begin{array}{c}
		{\sf \Sigma}_{\boldsymbol{C}}  \\
		\hline
		0  \end{array}\right)
	\equiv
	\left(\begin{array}{cc}
		{\sf \Sigma}_{t}&0  \\
		0&{\Sigma}_{s}\epsilon   \\
		\hline
		0&0  \\
	\end{array}\right)
	=\left(\begin{array}{ccccc}
		\frac{\sqrt{2}}{2}+\epsilon&0&0&0&0\\
		0&\frac{\sqrt{2}}{2}&0&0&0\\
		0&0&\frac{\sqrt{2}}{2}-\epsilon&0&0\\
		0&0&0&\epsilon&0\\
		0&0&0&0&\epsilon\\
		\hline
		0&0&0&0&0
	\end{array}\right),
\end{equation}
from which we have  ${\sf \Sigma}_{\boldsymbol{C}} ={\rm diag}({\sf \Sigma}_{t},{\Sigma}_{s}\epsilon)$ and
\begin{equation}\label{Sigts}
	{\sf \Sigma}_{t} ={\rm diag}\left(\frac{\sqrt{2}}{2}+\epsilon, \frac{\sqrt{2}}{2}, \frac{\sqrt{2}}{2}-\epsilon\right), \quad {\Sigma}_{s} =I_2.
\end{equation}
Here, we have used \eqref{SVDP} with $s=2,t=3$.	
Partition $\boldsymbol{P}$ as the following 2-by-3 blocked matrix
\begin{equation*}
	\boldsymbol{P}=\begin{pmatrix}
		\hat{\boldsymbol{P}}_{11}&  \hat{\boldsymbol{P}}_{12}& \hat{\boldsymbol{P}}_{13}\\
		\hat{\boldsymbol{P}}_{21}& \hat{\boldsymbol{P}}_{22}& \hat{\boldsymbol{P}}_{23}
	\end{pmatrix}=
	\left(\begin{array}{ccc|cc|c}
		1&(\frac{\sqrt{2}}{2}{\bf k}-1)\epsilon&0&0&-\epsilon&-\frac{\sqrt{2}}{2}{\bf k}\epsilon\\
		({\bf k}+\frac{\sqrt{2}}{2})\epsilon&\frac{\sqrt{2}}{2}&0&0&0&-\frac{\sqrt{2}}{2}\\
		0&0&\epsilon&1&0&0\\
		\hline
		\epsilon&\frac{\sqrt{2}}{2}{\bf j}\epsilon&0&0&1&\frac{\sqrt{2}}{2}{\bf j}\epsilon\\
		\frac{\sqrt{2}}{2}\epsilon&\frac{\sqrt{2}}{2}&0&0&{\bf j}\epsilon&\frac{\sqrt{2}}{2}\\
		0&0&1&-\epsilon&0&0\\
	\end{array}\right).
\end{equation*}
From \cref{coroCS}  the DQCS decomposition of $\begin{pmatrix}
	\hat{\boldsymbol{P}}_{11}\\
	\hat{\boldsymbol{P}}_{21}
\end{pmatrix}$ 
is derived as follows
\begin{equation}\label{Ex1P12}
	\begin{aligned}	
		\begin{pmatrix}
			\hat{\boldsymbol{P}}_{11}\\
			\hat{\boldsymbol{P}}_{21}
		\end{pmatrix}
		&=\begin{pmatrix}
			\hat{\boldsymbol{U}}& 0  \\
			0  & \hat{\boldsymbol{V}}
		\end{pmatrix}\begin{pmatrix}
			\hat{{\sf \Sigma}}_{\boldsymbol{A}}\\
			\hat{{\sf \Sigma}}_{\boldsymbol{B}}
		\end{pmatrix}\hat{\boldsymbol{W}}^{*}\\
		&=\left(\begin{array}{ccc|ccc}
			1&{\bf k}\epsilon&0&0&0&0\\
			{\bf k}\epsilon&1&0&0&0&0\\
			0&0&1&0&0&0\\
			\hline
			0&0&0&1&{\bf j}\epsilon&0\\
			0&0&0&{\bf j}\epsilon&1&0\\
			0&0&0&0&0&1
		\end{array}\right)
		\begin{pmatrix}
			1&0&0\\
			0&\frac{\sqrt{2}}{2}&0\\
			0&0&\epsilon\\
			\hline
			\epsilon&0&0\\
			0&\frac{\sqrt{2}}{2}&0\\
			0&0&1
		\end{pmatrix}\begin{pmatrix}
			1&-\epsilon&0\\
			\epsilon&1&0\\
			0&0&1
		\end{pmatrix}.
	\end{aligned}
\end{equation}
Then, from \eqref{XXhat}	 we obtain a nonsingular dual quaternion matrix
\begin{equation}\label{ex1X}
	\hat{\boldsymbol{X}}={\rm diag}(\hat{{\boldsymbol{W}}}^*{\sf \Sigma}_{t},\Sigma_{s},I_{n-k})\boldsymbol{Q}^{*}=\begin{pmatrix}
		\frac{\sqrt{2}}{2}+\epsilon&0&0&0&0\\
		0&\frac{\sqrt{2}}{2}&0&0&0\\
		0&0&\frac{\sqrt{2}}{2}-\epsilon&\frac{\sqrt{2}}{2}\epsilon&0\\
		0&0&-\epsilon&1&0\\
		0&0&0&0&1
	\end{pmatrix},  
\end{equation}
and
\begin{subequations}\label{ex1NANB}
	\begin{align}
		\boldsymbol{N}_{\boldsymbol{A}}&=\hat{\boldsymbol{U}}^\ast \hat{\boldsymbol{P}}_{12}=\left(\begin{array}{ccc}
			1&{\bf k}\epsilon&0\\
			{\bf k}\epsilon&1&0\\
			0&0&1
		\end{array}\right)^\ast \left(\begin{array}{cc}
			0&-\epsilon\\
			0&0\\
			1&0
		\end{array}\right)=\begin{pmatrix}
			0&-\epsilon\\
			0&0\\
			1&0
		\end{pmatrix},  \\
		\boldsymbol{N}_{\boldsymbol{B}}&=\hat{\boldsymbol{V}}^\ast \hat{\boldsymbol{P}}_{22}=
		\left(\begin{array}{ccc}
			1&{\bf j}\epsilon&0\\
			{\bf j}\epsilon&1&0\\
			0&0&1
		\end{array}\right)^\ast 
		\left(\begin{array}{cc}
			0&1\\
			0&{\bf j}\epsilon  \\
			-\epsilon&0  \\
		\end{array}\right)=\begin{pmatrix}
			0&1\\
			0&0\\
			-\epsilon&0
		\end{pmatrix}.
	\end{align}
\end{subequations}

Combining \eqref{ex1X}-\eqref{ex1NANB} with the submatrices in \eqref{Ex1P12} we obtain the DQGSVD1 of $\{\boldsymbol{A}, \boldsymbol{B}\}$
in the form of \eqref{eq:AB2}, that is,
\begin{equation*}
	\begin{aligned}	
	\boldsymbol{A}&=\hat{\boldsymbol{U}}(\hat{\sf \Sigma}_{\boldsymbol{A}},  \boldsymbol{N}_{\boldsymbol{A}}\epsilon)\hat{\boldsymbol{X}}\\
	&=\left(\begin{array}{ccc}
		1&{\bf k}\epsilon&0\\
		{\bf k}\epsilon&1&0\\
		0&0&1
	\end{array}\right)
	\left(\begin{array}{ccc|cc}
		1&0&0&0&0\\
		0&\frac{\sqrt{2}}{2}&0&0&0\\
		0&0&\epsilon&\epsilon &0
	\end{array}\right)\begin{pmatrix}
		\frac{\sqrt{2}}{2}+\epsilon&0&0&0&0\\
		0&\frac{\sqrt{2}}{2}&0&0&0\\
		0&0&\frac{\sqrt{2}}{2}-\epsilon&\frac{\sqrt{2}}{2}\epsilon&0\\
		0&0&-\epsilon&1&0\\
		0&0&0&0&1
	\end{pmatrix},
	\end{aligned}
\end{equation*}
	\begin{equation*}
		\begin{aligned}
			\boldsymbol{B}&=\hat{\boldsymbol{V}}(\hat{{\sf \Sigma}}_{\boldsymbol{B}},\boldsymbol{N}_{\boldsymbol{B}}\epsilon)\hat{\boldsymbol{X}}\\
			&=\left(\begin{array}{ccc}
			1&{\bf j}\epsilon&0\\
			{\bf j}\epsilon&1&0\\
			0&0&1
		\end{array}\right)
		\left(\begin{array}{ccc|cc}
			\epsilon&0&0&0&\epsilon\\
			0&\frac{\sqrt{2}}{2}&0&0&0\\
			0&0&1&0 &0
		\end{array}\right)\begin{pmatrix}
			\frac{\sqrt{2}}{2}+\epsilon&0&0&0&0\\
			0&\frac{\sqrt{2}}{2}&0&0&0\\
			0&0&\frac{\sqrt{2}}{2}-\epsilon&\frac{\sqrt{2}}{2}\epsilon&0\\
			0&0&-\epsilon&1&0\\
			0&0&0&0&1
		\end{pmatrix}.
		\end{aligned}
	\end{equation*}


In the following example, we verify the principle of DQGSVD2, another quotient-type singular value decomposition of the dual quaternion matrix pair $\{{\boldsymbol A}, {\boldsymbol B}\}$. 

\begin{example}\label{examp3}	
	Find DQGSVD2 of the dual quaternion matrix pair $\{\boldsymbol{A}, \boldsymbol{B}\}$, where $\boldsymbol{A}$ and $\boldsymbol{B}$ are given in \eqref{ex2AB2}.
\end{example}

{\it Solution.} We begin with the DQSVD of $\boldsymbol{C}=\begin{pmatrix}
	\boldsymbol{A}\\
	\boldsymbol{B}
\end{pmatrix}$ 	
in \eqref{Ex1DQsvd}. Recalling $s=2, t=3$, partition the matrix $\boldsymbol{Q}$ in \eqref{Ex2Q} to be
$\boldsymbol{Q}=(\boldsymbol{Q}_{1},\ \boldsymbol{Q}_{2})$ with $\boldsymbol{Q}_{1}\in{\bf \mathbb{DQ}}^{5 \times 3}$, $\boldsymbol{Q}_{2}\in{\bf \mathbb{DQ}}^{5 \times 2}$, that is,
\begin{equation}\label{Ex3Q}
	\boldsymbol{Q}=(\boldsymbol{Q}_{1},\ \boldsymbol{Q}_{2})=\left(\begin{array}{ccc|cc}
		1&-\epsilon&0&0&0\\
		\epsilon&1&0&0&0\\
		0&0&1&-\epsilon&0\\
		0&0&\epsilon&1&0\\
		0&0&0&0&1
	\end{array}
	\right).
\end{equation} 

Combining \eqref{Sigts} and from \eqref{DQGSVDeq1} and \eqref{DQGSVDeq2} we have
\begin{equation*}
	\begin{aligned}
		\boldsymbol{A}_{1}&=\boldsymbol{A}\boldsymbol{Q}_{1}{\sf \Sigma}_{t}^{-1}=
		\begin{pmatrix}
			1&(\frac{\sqrt{2}}{2}{\bf k}-1)\epsilon&0\\
			({\bf k}+\frac{\sqrt{2}}{2})\epsilon&\frac{\sqrt{2}}{2}&0\\
			0&0&\epsilon\\
		\end{pmatrix},\\	
		\boldsymbol{B}_{1}&=\boldsymbol{B}\boldsymbol{Q}_{1}{\sf \Sigma}_{t}^{-1}=
		\begin{pmatrix}
			\epsilon&\frac{\sqrt{2}}{2}{\bf j}\epsilon&0\\
			\frac{\sqrt{2}}{2}\epsilon&\frac{\sqrt{2}}{2}&0\\
			0&0&1\\
		\end{pmatrix}.
	\end{aligned}
\end{equation*}

From \eqref{DQGSVDeq3} the DQSVD on $\boldsymbol{A}_{1}$ gives
\begin{equation*}
	\hat{\boldsymbol{U}}^{*}{\boldsymbol{A}_{1}}{\boldsymbol{W}}=\hat{{\sf \Sigma}}_{\boldsymbol{A}}=
	\begin{pmatrix}
		1&0&0\\
		0&\frac{\sqrt{2}}{2}&0\\
		0&0&\epsilon\\
	\end{pmatrix}
\end{equation*}
where
\begin{equation}\label{Ex3hatUW}
	\hat{\boldsymbol{U}}=\begin{pmatrix}
		1&{\bf k}\epsilon&0\\
		{\bf k}\epsilon&1&0\\
		0&0&1
	\end{pmatrix},\quad 
	\boldsymbol{W}=\begin{pmatrix}
		1&\epsilon&0\\
		-\epsilon&1&0\\
		0&0&1
	\end{pmatrix}.
\end{equation}

From \eqref{Bgsv} there exists a unitary dual quaternion matrix 
\begin{equation*}
	\hat{\boldsymbol{V}}=\begin{pmatrix}
		1&{\bf j}\epsilon&0\\
		{\bf j}\epsilon&1&0\\
		0&0&1
	\end{pmatrix}
\end{equation*}
such that
\begin{equation}\label{Ex3Bgsv}
	\hat{\boldsymbol{V}}^{*}\boldsymbol{B}_{1}\boldsymbol{W}=\begin{pmatrix}
		\epsilon&0&0\\
		0&\frac{\sqrt{2}}{2}&0\\
		0&0&1
	\end{pmatrix}.
\end{equation}

Recalling the formula \eqref{Bgsv}, we have ${\bf T}=1$ in \eqref{Ex3Bgsv}. As a result, we can omit steps \eqref{Tsvd}-\eqref{PB1W}, and directly overwrite 
\begin{equation*}
	\boldsymbol{U}=\hat{\boldsymbol{U}}, \quad \boldsymbol{V}=\hat{\boldsymbol{V}}
\end{equation*}
in \eqref{GsvXU} and ${\sf \Sigma}_{\boldsymbol{A}}=\hat{{\sf \Sigma}}_{\boldsymbol{A}}$ in \eqref{pai}, and obtain a singular dual quaternion matrix
\begin{equation*}
	\boldsymbol{X}=\boldsymbol{Q}\begin{pmatrix}
		{\sf {\Sigma}}_{t}^{-1}\boldsymbol{W}& 0 \\
		0 & 0  \\
	\end{pmatrix}=\begin{pmatrix}
		\sqrt{2}-2\epsilon&0&0&0&0\\
		0&\sqrt{2}&0&0&0\\
		0&0&\sqrt{2}+2\epsilon&0&0\\
		0&0&\sqrt{2}\epsilon&0&0\\
		0&0&0&0&0
	\end{pmatrix}
\end{equation*}
with the aid of $\boldsymbol{Q}$ in \eqref{Ex3Q}, $\boldsymbol{W}$ in \eqref{Ex3hatUW} and ${\sf {\Sigma}}_{t}$ in \eqref{Sigts}.
The dual matrix on the right-hand side of \eqref{Ex3Bgsv} gives ${\sf \Sigma}_{\boldsymbol{B}}$.

In summary, we obtain the DQGSVD2 of $\{\boldsymbol{A}, \boldsymbol{B}\}$
in the form of \eqref{eq:AB3}, that is,
\begin{equation*}
	\begin{aligned}
		\boldsymbol{U}^{*}\boldsymbol{A}\boldsymbol{X} &=
		\begin{pmatrix}
			1&{\bf k}\epsilon&0\\
			{\bf k}\epsilon&1&0\\
			0&0&1
		\end{pmatrix}^\ast 
		\begin{pmatrix}
			\frac{\sqrt{2}}{2}+\epsilon&\frac{1}{2}{\bf k}\epsilon&0&0&0\\
			\frac{\sqrt{2}}{2}{\bf k}\epsilon&\frac{1}{2}&0&0&0\\
			0&0&\frac{\sqrt{2}}{2}\epsilon&\epsilon&0
		\end{pmatrix}
		\begin{pmatrix}
			\sqrt{2}-2\epsilon&0&0&0&0\\
			0&\sqrt{2}&0&0&0\\
			0&0&\sqrt{2}+2\epsilon&0&0\\
			0&0&\sqrt{2}\epsilon&0&0\\
			0&0&0&0&0
		\end{pmatrix}  \\
		&=\left(\begin{array}{ccc|cc}
			1&0&0&0&0\\
			0&\frac{\sqrt{2}}{2}&0&0&0\\
			0&0&\epsilon&0&0\\
		\end{array}\right) = ({\sf \Sigma}_{\boldsymbol{A}}, 0),  \\
		\boldsymbol{V}^{*}\boldsymbol{B}\boldsymbol{X} &=
		\begin{pmatrix}
			1&{\bf j}\epsilon&0\\
			{\bf j}\epsilon&1&0\\
			0&0&1
		\end{pmatrix}^\ast 
		\begin{pmatrix}
			\frac{\sqrt{2}}{2}\epsilon&\frac{1}{2}{\bf j}\epsilon&0&0&\epsilon\\
			0&\frac{1}{2}&0&0&0\\
			0&0&\frac{\sqrt{2}}{2}-\epsilon&\frac{\sqrt{2}}{2}\epsilon&0
		\end{pmatrix}
		\begin{pmatrix}
			\sqrt{2}-2\epsilon&0&0&0&0\\
			0&\sqrt{2}&0&0&0\\
			0&0&\sqrt{2}+2\epsilon&0&0\\
			0&0&\sqrt{2}\epsilon&0&0\\
			0&0&0&0&0
		\end{pmatrix}  \\
		&=\left(\begin{array}{ccc|cc}
			\epsilon&0&0&0&0\\
			0&\frac{\sqrt{2}}{2}&0&0&0\\
			0&0&1&0&0\\
		\end{array}\right) = ({\sf \Sigma}_{\boldsymbol{B}}, 0).
	\end{aligned}
\end{equation*}

\section{Conclusion}\label{section6}

In this paper we have investigated the Householder transformation, QR decomposition, and DQCS decomposition of dual quaternion matrices. Based on these theoretical results, we mainly focus on the generalized singular value decomposition of a dual quaternion matrix pair in accordance with their dimensions. According to the rank or appreciable rank of the dual quaternion matrices, we have derived two forms of the DQGSVD, which afford significant importance for deeply understanding the properties of dual quaternion matrices and their applications in related fields. Except for these,
we have proposed the DQPSVD and the DQCCD decompositions of dual quaternion matrix pair yet. The various types of DQGSVD decompositions of the dual quaternion matrix triplet $\{{\boldsymbol A},{\boldsymbol B}, {\boldsymbol C}\}$ are our ongoing research work. We are committed to sharing the corresponding theoretical results in a forthcoming paper. We believe that the theoretical results derived in this paper have great potential applications for solving problems related to dual quaternion matrix equations, such as robot hand-eye calibration problem and others, which deserve our further investigation.

\section*{Acknowledgements}
This work was supported by the Fundamental Research Funds for the Central Universities
	(Grant No. 2024ZDPYCH1004).

\







\end{document}